\documentclass[11pt]{amsart}
\usepackage{amsmath,amsfonts,amssymb,amsthm}
\usepackage{mathrsfs}
\usepackage[left=1.0in,right=1.0in,top=1.0in,bottom=1.0in]{geometry}

\numberwithin{equation}{section}

\newtheorem{theorem}{Theorem}[section]
\newtheorem{lemma}[theorem]{Lemma}
\newtheorem{corollary}[theorem]{Corollary}
\newtheorem{proposition}[theorem]{Proposition}
\newtheorem{claim}{Claim}

\theoremstyle{definition}

\newtheorem{example}[theorem]{Example}
\newtheorem{definition}[theorem]{Definition}
\newtheorem{remark}[theorem]{Remark}
\newtheorem{question}[theorem]{Question}

\newcommand{\T}{{\mathbb T}}
\newcommand{\Z}{{\mathbb Z}}
\newcommand{\N}{{\mathbb N}}
\newcommand{\R}{{\mathbb R}}
\newcommand{\Q}{{\mathbb Q}}
\newcommand{\cont}{{\mathfrak c}}
\newcommand{\Cp}[2]{C_p(#1,#2)}
\newcommand{\W}{\mathscr{W}}
\newcommand{\Wi}{\mathscr{W}_\infty}
\newcommand{\lma}[1]{\mathrm{Max}\,{#1}}
\newcommand{\lmi}[1]{\mathrm{Min}\,{#1}}
\newcommand{\ext}[1]{\mathrm{Ext}\,{#1}}
\newcommand{\sequence}{sequence with distinct neighbors}
\newcommand{\Seq}{{\mathfrak S}}
\newcommand{\support}[1]{\sigma\!_{#1}}
\newcommand{\nat}[1]{\underline{#1}}
\newcommand{\cancel}[2]{{#1\diamond#2}}
\newcommand{\letter}[2]{{#1}\lfloor{#2}\rfloor}
\newcommand{\power}[2]{{#1}\lceil{#2}\rceil}

\title{Productivity of sequences with respect to a given weight function}
\author[D. Dikranjan]{Dikran Dikranjan}
\address[Dikran Dikranjan]{Dipartimento di Matematica e Informatica, Universit\`{a} di Udine, Via delle Scienze  206, 33100 Udine, Italy}
\email{dikran.dikranjan@uniud.it}

\author[D. Shakhmatov]{Dmitri Shakhmatov}
\address[Dmitri Shakhmatov]{Division of Mathematics, Physics and Earth Sciences\\
Graduate School of Science and Engineering\\
Ehime University, Matsuyama 790-8577, Japan}
\email{dmitri.shakhmatov@ehime-u.ac.jp}

\author[J. Sp\v{e}v\'ak]{Jan Sp\v{e}v\'ak}
\address[Jan Sp\v{e}v\'ak]{Department of Mathematics\\ Faculty of Science\\ J. E. Purkinje
University,
\v{C}esk\'{e} ml\'{a}de\v{z}e 8,
400 96 \'{U}st\'{i} nad Labem\\
Czech Republic}
\email{jan.spevak@ujep.cz}

\dedicatory{In memory of Nigel Kalton}



\thanks{The first named author was partially supported by SRA grants P1-0292-0101 and J1-9643-0101, as well as grant MTM2009-14409-C02-01.}

\thanks{The second named author was partially supported by the Grant-in-Aid for Scientific Research
(C) no.~22540089 by the Japan Society for the Promotion of Science (JSPS)}

\thanks{Preliminary version of this manuscript was written during
the third named author's Ph.~D. study at the Graduate School of
Science and Engineering of Ehime
University (Matsuyama, Japan) from October 1, 2006 till September 30, 2009.
His stay at Ehime University during this period has been generously supported
by the Ph.~D. Scholarship of the Government of Japan.}

\begin{document}

\begin{abstract}
Given a function $f:\N\to(\omega+1)\setminus\{0\}$,
we say that a faithfully indexed sequence $\{a_n:n\in\N\}$ of elements of a topological group $G$ is:
(i)~{\em $f$-Cauchy productive ($f$-productive)\/} provided that the sequence $\{\prod_{n=0}^m
a_n^{z(n)}:m\in\N\}$ is left Cauchy (converges to some element of $G$, respectively) for each
function $z:\N\to\Z$ such that $|z(n)|\le f(n)$ for every $n\in\N$;
(ii)~{\em unconditionally $f$-Cauchy productive (unconditionally $f$-productive)\/} provided that the sequence $\{a_{\varphi(n)}:n\in\N\}$ is
$(f\circ\varphi)$-Cauchy productive (respectively, $(f\circ\varphi)$-productive)
for every bijection $\varphi:\N\to\N$. (Bijections can be replaced by injections here.)
We consider the question of existence of (unconditionally) $f$-productive sequences for a given ``weight function'' $f$. We prove that: (1)~a
Hausdorff group having an $f$-productive sequence for some $f$ contains a homeomorphic copy of the Cantor set;
(2)~if a non-discrete group is either locally compact Hausdorff  or Weil complete metric,
then it contains an unconditionally $f$-productive sequence for every function $f:\N\to\N$; (3)~a metric group is NSS if and only if
it does not contain an $f_\omega$-Cauchy productive sequence, where $f_\omega$ is the function taking the constant value $\omega$.
We give an example of an $f_\omega$-productive sequence $\{a_n:n\in\N\}$
 in a (necessarily non-abelian) separable metric group $H$ with a linear topology and a bijection $\varphi:\N\to\N$
such that the sequence $\{\prod_{n=0}^m a_{\varphi(n)}:m\in\N\}$ diverges, thereby answering a question of Dominguez and Tarieladze.
Furthermore, we show that $H$ has no unconditionally $f_\omega$-productive sequences.
As an application of our results, we resolve negatively a question from $C_p(-,G)$-theory.
\end{abstract}

\maketitle

{\em All topological groups considered in this paper are assumed to be Hausdorff.\/} 
The symbols
 $\Z$, $\Q$ and $\R$ denote the topological
groups of integer numbers, rational numbers and real numbers, and $\T$ denotes the quotient group $\R/\Z$.
We use $e$ to denote the identity element of a (topological) group $G$. As usual, $\N$ denotes the set of natural numbers, $\omega$ denotes the first infinite ordinal,
 $\omega+1$ denotes the successor of the ordinal $\omega$
and $\cont$ denotes the cardinality of the continuum.  In this paper, 
$0\in\N$.
An ordinal
is identified with the set of all smaller ordinals. 
In particular, $\omega$ is identified with $\N$ and $\omega+1$ is identified with $\N\cup\{\omega\}$.

\section{Introduction}

Recall that a series $\sum_{n=0}^\infty a_n$ of real numbers:
\begin{itemize}
\item[(i)]
 {\em converges\/} if the sequence $\left\{\sum_{n=0}^m a_n:n\in\N\/\right\}$ of partial sums converges to some real number,
\item[(ii)]
{\em converges unconditionally\/} if the series $\sum_{n=0}^\infty a_{\varphi(n)}$ converges for every bijection $\varphi: \N \to \N$, and
\item[(iii)]
{\em converges absolutely\/} if the series $\sum_{n=0}^\infty |a_n|$ converges.
\end{itemize}
According to the celebrated Riemann-Dirichlet theorem from the classical analysis, the series $\sum_{n=0}^\infty a_n$ converges
unconditionally if and only if it converges absolutely, and in such a case all series $\sum_{n=0}^\infty a_{\varphi(n)}$ converge to
the limit of the net  $\{\sum_{n\in F}a_n:F\in[\N]^{< \omega}\}$, independently on  the bijection $\varphi$.
(Here $[\N]^{< \omega}$ denotes the set of all finite subsets of $\N$ ordered by set inclusion.)

The first two of these classical notions can be easily extended for arbitrary abelian topological groups  \cite{B}. Let us say that a sequence $\{a_n:n\in\N\}$ in a topological abelian group $(G,+)$ is:
\begin{itemize}
\item[(1)] {\em summable\/} if the sequence $\left\{\sum_{n=0}^m a_n:m\in\N\/\right\}$ converges in $G$;
\item[(2)]  {\em unconditionally summable\/} if the sequence $\{a_{\varphi(n)}:n\in\N\}$ is summable for every bijection $\varphi: \N \to \N$.
\end{itemize}

It should be noted that the terminology adopted by various authors differs. In particular, our terminology in (1) and (2) differs from that used in \cite{B}.

Clearly, items (1) and (2) extend the notion of a convergent series (i) and that of an unconditionally convergent series (ii) in the
framework of topological abelian groups. It turns out that a sequence $\{a_n:n\in\N\}$ is unconditionally summable if and only if
the net $\{\sum_{n\in F}a_n:F\in[\N]^{< \omega}\}$ converges in $G$
\cite[III,
  \S 5.7, Prop. 9]{B}.

Another source of inspiration for the present paper is the classical notion of subseries convergence \cite{Kalton}; see  also
\cite{Swartz}. Recall that a series $\sum_{n=0}^\infty a_n$ in a topological group $G$ is said to be:
\begin{itemize}
\item[(iv)]
 {\em subseries convergent\/} provided that the series $\sum_{n=0}^\infty z_n a_n$ converges to an element of $G$ for every sequence $\{z_n:n\in\N\}$ of zeros and ones.
\end{itemize}

This concept plays a crucial role in the so-called Orlicz-Pettis type theorems and is a model case of the general notion of {\em multiplier convergence\/} of a series in functional analysis; see \cite{Swartz}. In this theory one typically takes a fixed set $\mathscr{F}\subseteq \R^\N$ of ``multipliers'' and calls a
series $\sum_{n=0}^\infty a_n$ in a topological vector space {\em $\mathscr{F}$-convergent\/} provided that the series $\sum_{n=0}^\infty f(n)a_n$ converges
for every $f\in\mathscr{F}$. By varying the set $\mathscr{F}$ of multipliers one can obtain a fine description of the level of convergence
of a given series, leading to a rich theory; see \cite{Swartz}.

Given the tremendous success of the multiplier convergence theory for topological vector spaces, it is a bit surprising that no systematic theory of multiplier convergence in {\em topological groups\/} was yet developed. This paper is an attempt to establish a foundation for such a theory. Quite naturally, when speaking of multipliers in a topological group, we have to restrict ourselves to a subset $\mathscr{F}$ of the set $\Z^\N$. Since an element of $\Z^\N$ can be viewed is a sequence of natural numbers, this leads to the following natural definition:

\begin{itemize}
\item[(3)]  Given a sequence $\{z_n:n\in\N\}$ of integer numbers, we say that a  sequence $\{a_n:n\in\N\}$ in an abelian topological group
 is {\em summable with multiplier $\{z_n:n\in\N\}$\/} if the sequence  $\{z_n a_n:n\in\N\}$ is summable.
 \end{itemize}

This notion allows to capture the notion of an absolutely convergent series (iii) in the framework of topological abelian groups. Indeed,
in case of a sequence $\{a_n:n\in\N\}$ of reals, for every $n\in \N$, let $\varepsilon_n=1$ if $a_n\ge 0$ and $\varepsilon_n=-1$ if
$a_n<0$. Then the series $\sum_{n=0}^\infty a_n$ converges absolutely if and only if the sequence $\{a_n:n\in\N\}$ is summable with the multiplier $\{\varepsilon_n:n\in\N\}$.
In this paper we use a single ``weight'' function $f:\N\to (\omega+1)\setminus\{0\}$ to define two multiplier sets
$$
\mathscr{F}_f=\{z\in \Z^\N: |z(n)|\le f(n)
\mbox{ for all }
n\in\N\}
\ \
\mbox{and}
\ \
\mathscr{F}^\star_f=\{z\in \Z^\N: 0\le z(n)\le f(n)
\mbox{ for all }
n\in\N\},
$$
as well as the corresponding notions of multiplier convergence. We call a one-to-one sequence $\{a_n:n\in\N\}$ of elements of a topological group:
\begin{itemize}
\item[(4a)] {\em $f$-summable\/}
({\em $f^\star$-summable})
if $\{a_n:n\in\N\}$ is summable with multiplier $\{z(n):n\in\N\}$ for every $z\in \mathscr{F}_f$ (for every $z\in \mathscr{F}^\star_f$, respectively);

\item[(4b)]
{\em unconditionally $f$-summable\/} if the sequence $\{a_{\varphi(n)}:n\in\N\}$ is $(f\circ \varphi)$-summable for every bijection $\varphi:\N\to\N$.
\end{itemize}

Replacing bijections with injections in (4b) leads to the same notion; see Lemma \ref{AP:set:has:AP:subsets}. All three notions
coincide in abelian groups; see Proposition \ref{fstar:is:f:for:abel}
and Corollary
\ref{abelian:corollary}. We construct an example distinguishing these notions
in non-abelian topological groups. In the abelian case, unconditionally summable sequences coincide with (unconditionally) $f_1$-summable sequences,
where $f_1$ is the constant function $f_1$ taking value $1$. It is clear that a one-to-one sequence $\{a_n:n\in\N\}$ is
$f_1^\star$-summable if and only if the series $\sum_{n=0}^\infty a_n$ is subseries convergent, so we cover also the classical notion (iv).

An ``exceptionally strong absolute mixture'' of (2) and (3) was introduced in \cite{Sp}. We give here a reformulation which is substantially different from, yet (non-trivially)
equivalent to the definition in \cite{Sp}, in case of an abelian group. We refer the reader to Remark \ref{reformulation:of:f-productivity} for details.

\begin{itemize}
\item[(5)]
A one-to-one sequence $\{a_n:n\in\N\}$ of elements of a topological abelian group is called an {\em absolutely summable set\/} if the sequence $\{z_n a_n:n\in\N\}$
is unconditionally summable for {\em every\/} weight sequence $\{z_n:n\in\N\}$ of integer numbers.
\end{itemize}

Clearly, an absolutely summable set is precisely an unconditionally $f_\omega$-productive sequence in our terminology,
where $f_\omega$ is the constant function taking value $\omega$. Therefore, we cover also (5) as a special case.

A topological group without infinite absolutely summable sets (unconditionally $f_\omega$-productive sequences, in our terminology)
is called a {\em TAP group\/} \cite{Sp}.  Clearly, a group $G$ is TAP whenever $G$ contains no nontrivial convergent sequences.
Moreover, every NSS group is TAP (\cite{Sp}; see also \cite{DT} or Theorem \ref{NSS:is:NACP} for a short proof of this result).
The converse implication holds for locally compact (not necessarily abelian) groups;
in particular, a locally compact TAP group is NSS, hence a Lie group \cite{DSS}.
The notion of a TAP group plays a key role in characterizing pseudocompactness of a topological space $X$ in terms of properties
of the topological group $C_p(X,G)$ of all $G$-valued continuous functions on $X$ endowed with the pointwise topology, for a given NSS
group $G$; see \cite{Sp} and Section \ref{sec:appl} for details.

The paper is organized as follows. In Section \ref{sec:(Cauchy) productive sequences} we recall the (absolute) properties of a (Cauchy) productive sequence and a (Cauchy) summable sequence. In Section \ref{sec:f-(Cauchy) productive sequences, for a given function f:N to omega+1} we introduce (Cauchy)
productivity (summability) for a sequence with respect to a weight function $f:\N\to(\omega+1)\setminus\{0\}$.
The unconditional versions of these notions are introduced in Section \ref{f-(Cauchy) productive sets, for a given function f: N to omega+1}.

In Section \ref{f-(Cauchy) productive sets in metric groups}  we investigate the existence of (unconditionally) $f$-(Cauchy) productive sequences
in metric groups. We prove that every non-discrete metric group contains an unconditionally $f$-Cauchy productive sequence
for every ``finitely-valued'' function $f:\N\to\N\setminus\{0\}$; see Theorem \ref{non1-TAP}. Consequently, every non-discrete Weil complete metric group has an
unconditionally $f$-productive sequence  for every function $f:\N\to\N\setminus\{0\}$ (Corollary \ref{no:f-product:seq:in:nondiscrete:metr:group}). We
show that a metric group $G$ is NSS if and only if it does not contain an
(unconditionally) $f_\omega$-Cauchy productive sequence (Theorem \ref{metric:NACP:iff:NSS}).

 In Section \ref{f-(Cauchy) productive sets in linear groups} we consider unconditionally $f$-(Cauchy) productive sequences
 in groups with a linear topology. In particular, we prove that a linear sequentially  Weil complete group $G$ is TAP if and only if
it contains no nontrivial convergent sequences (Theorem \ref{linear:TAP}),
and if $G$ is also sequential, then $G$ is TAP if and only if it is discrete (Corollary \ref{sequentially:complete:liner:TAP:group:is:discrete}).

In Section \ref{f-productive sets in locally compact and omega-bounded groups} we show that every non-discrete
group $G$ that is either locally compact or $\omega$-bounded contains an unconditionally
$f$-productive sequence for every ``finite-valued'' weight function $f:\N\to\N\setminus\{0\}$;
moreover, if $G$ is  also totally disconnected, then $G$ contains even an unconditionally
$f_\omega$-productive sequence.

Section \ref{Descriptive properties of groups having f_1star-productive sequences} contains one of our main results: a topological group having an $f$-productive sequence for some $f$ contains a homeomorphic copy of the Cantor set. In particular, every topological group of size $< \cont$ is TAP (Corollary \ref{|G|<cont:is:TAP}).

In Section \ref{sec:appl} we give an application of our results to group-valued function spaces by resolving an open question from \cite{Sp}.

In Section \ref{four:spectra:section} we introduce four different notions of productivity spectra of a given topological group $G$ related to the existence of (unconditionally)
$f$-productive and (unconditionally) $f$-Cauchy productive sequences in $G$, and we show that there are precisely four classes
of abelian topological groups with respect to these spectra:
\begin{itemize}
\item
groups that have no $f$-productive sequences for any $f$,
\item
groups that have $f$-productive sequences only for bounded weight functions $f$,
\item
groups that have $f$-productive sequences only for ``finite-valued'' weight functions $f:\N\to\N$, and
\item
groups that have $f$-productive sequences for every weight function $f:\N\to(\omega+1)\setminus\{0\}$.
\end{itemize}
Moreover, we provide examples of separable metric abelian groups witnessing that these four classes are distinct.

In Section \ref{Eliminating f-productive sequences in products} we develop a general technique for eliminating $f$-productive sequences in infinite products
of discrete groups based on seminorms; see Lemma \ref{impossible:case:for:unbounded:seminorm}.

In Section \ref{background:on:free:groups} we recall a necessary background on free groups
and prove the ``mountain lemma''; see Lemma \ref{real:final:step}. This lemma is merely a technical tool used only in the last Section \ref{proof:of:item:v} dedicated to a tedious proof
of item (v) of Theorem \ref{TAP:group:with:ap:sequence}. Readers who are not interested in that proof may skip reading this ``mountain lemma''.

In Section \ref{seminorms:from:linear:order} we define a seminorm on free groups
that arise from a linear order on their alphabet set; see Definition \ref{eta:definition}
and Lemma \ref{a(g):equation:2}.
This seminorm appears to be well suitable for capturing the non-commutative nature of free groups. It is a main technical tool
used in Sections \ref{Sec:f-prod:seq:not:f-prod:set} and \ref{proof:of:item:v}.

In Section \ref{Sec:f-prod:seq:not:f-prod:set} we construct a  (necessarily non-abelian) separable metric linear group $H$
that contains an $f_\omega$-productive sequence $A=\{a_n:n\in\N\}$ with an additional property that the sequence $\{\prod_{n=0}^m a_{\varphi(n)}:m\in\N\}$ diverges
for a suitable bijection $\varphi:\N\to\N$. This provides a negative answer to a question of Dominguez and Tarieladze from \cite[Remark 2.5]{DT-private}.

Finally, in the last Section \ref{proof:of:item:v} we use the full power of the sophisticated machinery developed in Sections \ref{Eliminating f-productive sequences in products}--\ref{Sec:f-prod:seq:not:f-prod:set} in order to prove that the group $H$ does not contain any unconditionally $f_\omega$-productive sequence;  that is, $H$ is TAP.

\section{(Cauchy) productive sequences}
\label{sec:(Cauchy) productive sequences}

Recall that a sequence $\{a_n:n\in\N\}$ of elements of a topological group $G$ is called {\em left Cauchy\/}  ({\em right Cauchy\/}) provided that for every neighborhood $U$ of the
identity $e$ of $G$ there exists $k\in\N$ such that $a_l^{-1}a_m\in U$ (respectively, $a_la_m^{-1}\in U$) whenever $l,m$ are integers satisfying $l\geq k$ and $m\geq k$.
Recall that a topological group $G$ is called {\em sequentially Weil complete\/} provided that every left Cauchy sequence in $G$ converges to some element of $G$.

\begin{proposition}
For a topological group $G$, the following statements are equivalent:
\begin{itemize}
\item[(i)] $G$ is sequentially Weil complete;
\item[(ii)] every right Cauchy sequence in $G$ converges;
\item[(iii)] every left Cauchy sequence in $G$ is also a right Cauchy sequence  in $G$, and $G$ is sequentially closed
in its Raikov completion (that is, in the completion of $G$ with respect to the two-sided uniformity).
\end{itemize}
\end{proposition}

\begin{lemma}\label{G:with:no:conv:seq:is:seq:Weil:complete}
Every topological group without nontrivial convergent sequences is sequentially Weil complete.
\end{lemma}

\begin{proof}
Let $G$ be a topological group without nontrivial convergent sequences, and assume that $\{a_n:n\in\N\}$ is a left Cauchy
sequence in $G$. Then for every neighborhood $U$ of the identity $e$ of $G$ there exists $k\in\N$ such that $a_n^{-1}a_{n+1}\in U$
for every integer $n$ such that $n\ge k$. Thus, the sequence $\{a_n^{-1}a_{n+1}:n\in\N\}$ converges to $e$. By our assumption,
there exists $i\in\N$ such that $a_n^{-1}a_{n+1}=e$ for every $n\ge i$. Therefore, $a_n=a_i$ for every $n\ge i$, and so the sequence $\{a_n:n\in\N\}$ converges to $a_i$.
\end{proof}

\begin{definition}
\label{productive}
A sequence $\{b_n:n\in\N\}$ of elements of a topological group $G$ will be called:
\begin{itemize}
\item[(i)]
{\em Cauchy productive\/} provided that the sequence $\left\{\prod_{n=0}^m b_n:m\in\N\/\right\}$ is left Cauchy;
\item[(ii)]
{\em productive\/} provided that the
sequence $\left\{\prod_{n=0}^m b_n:m\in\N\/\right\}$ converges to some element of $G$.
\end{itemize}
\end{definition}

When the group $G$ is abelian, we will use additive notations:
\begin{definition}
\label{summable}
A sequence $\{b_n:n\in\N\}$ of elements of an abelian topological group $G$ will be called:
\begin{itemize}
\item[(i)]
{\em Cauchy summable\/} provided that the sequence
$\left\{\sum_{n=0}^m b_n:m\in\N\/\right\}$ is a Cauchy sequence (in any of the three coinciding uniformities on $G$);
\item[(ii)]
{\em summable\/} provided that the sequence $\left\{\sum_{n=0}^m b_n:m\in\N\/\right\}$ converges to some element of $G$.
\end{itemize}
\end{definition}

\begin{lemma}
\label{productive:and:Cauchy:productive:coincide}
A sequence of elements of a sequentially Weil complete (abelian) group is Cauchy productive (summable) if and only if it is productive (summable).
\end{lemma}

The following lemma is an analogue of the Cauchy criterion for convergence of a series in the real line.

\begin{lemma}
\label{Cauchy:productive:criterion}
For a sequence $B=\{b_n:n\in\N\}$ of elements of a topological group $G$ the following conditions are equivalent:
\begin{itemize}
\item[(i)] $B$ is Cauchy productive;
\item[(ii)] for every open neighborhood $U$ of $e$ there exists $k\in\N$ such that $\prod_{n=l}^m b_n\in U$ whenever $l,m\in\N$ and $k\le l\le m$.
\end{itemize}
\end{lemma}

\begin{corollary}
Every Cauchy productive sequence in a topological group converges to the identity element.
\end{corollary}

\begin{corollary}\label{no:sequences:no:Cauchy:prod:sequences}
A topological group without nontrivial convergent sequences has no nontrivial Cauchy productive sequences.
\end{corollary}

\begin{example}
\label{symmetric:group:example}
 Let  $G =\{f:\N\to\N\ |\ f$ is a bijection$\}\subseteq\N^\N$ be the symmetric group with the usual
 pointwise convergence topology. For each $n\in\N$, denote by $b_n$ the transposition
of $n$ and $n +1$. Let $B=\{b_n:n\in\N\}\subseteq G$. Then $G$ is a metric group with the following properties:
\begin{itemize}
\item[(i)] $G$ is two-sided (Raikov) complete;
\item[(ii)] $G$ is not Weil complete;
\item[(iii)] the sequence $B$ converges to the identity element $id$ of $G$;
\item[(iv)] the sequence $B$ is Cauchy productive;
\item[(v)] the sequence $B$ is not productive.
 \end{itemize}
 Indeed, (i) and (ii) are well known; see \cite[Chapter 7]{DPS}. To check (iii), it suffices to note
 that for every $m\in\N$ and each integer $n>m$ one has $b_n(k) = k = id(k)$
 for all $k\in\N$ with $k \leq m$. This proves that $b_n$ converges to the identity element $id$ of $G$.

Let us prove (iv). Given $k,l,m,\in\N$ with $k\le l\le m$, one can easily see that  $b_{l} b_{l+1}\dots b_m(i)=i$ for any $i\in\N$ with
$i<k$. Since $G$ has the topology of pointwise convergence, Lemma \ref{Cauchy:productive:criterion} yields that $B$ is Cauchy productive.

To prove (v), let $\pi_n= b_0 b_1 \ldots b_n$ and note that
\begin{equation}
\label{pi:is:a:cycle}
\mbox{$\pi_n$ is the cycle $0\to 1\to 2\to\dots\to n\to n+1\to 0$ with support $\{0,1, \ldots, n, n+1\}$.}
\end{equation}
Assume that the sequence $\{\pi_n:n\in\N\}$ converges to some $\sigma \in G$. Since $\sigma$ is surjective, there exists $k \in \N$ such that  $1 = \sigma(k)$.
Since $G$ has the topology of pointwise convergence,  from $\pi_n\to\sigma$  it follows that  there exists an integer $n_0 > k$ such that
\begin{equation}
\label{eq:dagger}
 \pi_n(k) = \sigma(k)=1
\mbox{ for all integers }
n>n_0.
\end{equation}
On the other hand, $\pi_n(k) = k +1  \ne 1$ for $n>n_0$ by \eqref{pi:is:a:cycle}, which contradicts \eqref{eq:dagger}.   Hence, $B$ is not productive.
\end{example}

\begin{lemma}\label{sum:of:lin:is:lim:of:sum}
If $\{a_n:n\in\N\}$ and $\{b_n:n\in\N\}$ are (Cauchy) summable sequences in an abelian topological group $G$,
then the sequence $\{a_n+b_n:n\in\N\}$ is also (Cauchy) summable in $G$.
\end{lemma}

\begin{proof}
Suppose that $\{a_n:n\in\N\}$ and $\{b_n:n\in\N\}$ are Cauchy summable in $G$. Let $U$ be an open neighborhood of $0$.
Choose an open neighborhood $V$ of $0$ such that $V+V\subseteq U$. By Lemma \ref{Cauchy:productive:criterion}, we can
find $k\in\N$ such that $\sum_{n=l}^m a_n\in V$ and $\sum_{n=l}^m b_n\in V$ whenever $l$ and $m$ are integers satisfying $k\le l\le m$. Since $G$ is abelian,
$\sum_{n=l}^m (a_n+b_n)= \sum_{n=l}^m a_n + \sum_{n=l}^m b_n \in V+V\subseteq U$.
Applying Lemma \ref{Cauchy:productive:criterion} once again, we conclude that the sequence $\{a_n+b_n:n\in\N\}$ is Cauchy summable.

A similar proof for summable sequences is left to the reader.
\end{proof}

\section{$f$-(Cauchy) productive sequences, for a given weight function $f$}
\label{sec:f-(Cauchy) productive sequences, for a given function f:N to omega+1}

\begin{definition}
\begin{itemize}
\item[(i)]
Denote by $\W_{\infty}$ the set of all functions $f:\N\to(\omega+1)\setminus\{0\}$, and let $\W$ be the subset of all $f\in \W_\infty$ such that $f(n)\not=\omega$ for all $n\in\N$.
\item[(ii)]
Let $f, g\in \Wi$.
\begin{itemize}
\item[(a)] $f\le g$ means that $f(n)\le g(n)$ for every $n\in\N$.
\item[(b)] $f\le^* g$ means that there exists $i\in\N$ such that $f(n)\le g(n)$ for all integers $n$ satisfying $n\ge i$. Similarly, $f=^* g$ means that there exists $i\in\N$ such that $f(n)= g(n)$ for all integers $n$ satisfying $n\ge i$.
\end{itemize}
\item[(iii)]
$f_\omega\in \Wi$ is the function defined by $f_\omega(n)=\omega$ for all $n\in\N$, and $f_1\in \Wi $ is the function defined by $f_1(n)=1$ for all $n\in\N$.
\end{itemize}
\end{definition}

Note that $f_\omega$ and $f_1$ are the biggest and the smallest element, respectively, of the poset $(\Wi, \leq)$.

\begin{definition}
For a function $z:\N\to\Z$ we denote by $|z|$ the function $f:\N\to\N$ defined by $f(n)=|z(n)|$ for all $n\in\N$.
\end{definition}

\begin{definition}\label{def:ap-sequence}
For $f\in \Wi$ we say that a faithfully indexed sequence $\{a_n:n\in\N\}$ of elements of a topological group $G$ is:
\begin{itemize}
\item[(i)]
{\em $f$-Cauchy productive ($f$-productive) in $G$\/} provided that the sequence $\{a_n^{z(n)}:n\in\N\}$ is Cauchy productive (respectively, productive) in $G$ for every function $z:\N\to\Z$ such that $|z|\le f$;
\item[(ii)]
{\em $f^\star$-Cauchy productive ($f^\star$-productive) in $G$\/} provided that the sequence $\{a_n^{z(n)}:n\in\N\}$ is Cauchy productive (respectively, productive) in $G$ for every function $z:\N\to\N$ such that $z\le f$.
\end{itemize}
\end{definition}

When the group $G$ is abelian, we will use the natural variant of the
above notions with ``productive'' replaced by ``summable''.  We will
also omit ``in $G$'' when the group $G$ in question is clear from the
context.

\begin{remark}
If $f\in\Wi$ and $H$ is a sequentially closed subgroup of a topological group $G$, then a sequence of elements of $H$ is $f$-productive in $G$ if and only if it is $f$-productive in $H$.
\end{remark}

\begin{lemma}
\label{easy:connections}
Let $f\in \Wi$ and let $A=\{a_n:n\in\N\}$ be a faithfully indexed sequence of elements of a topological group.
\begin{itemize}
\item[(i)]
If $A$ is $f$-productive, then $A$ is $f$-Cauchy productive.
\item[(ii)]
If one additionally assumes that $G$ is sequentially Weil complete, then $A$ is $f$-productive if and only if it is $f$-Cauchy productive.
\end{itemize}
\end{lemma}
\begin{proof}
(i) is clear, and (ii) follows from Lemma \ref{productive:and:Cauchy:productive:coincide} and Definition \ref{def:ap-sequence}.
\end{proof}

\begin{lemma}\label{left:cauchy:condition}
Let
$f \in \Wi$.  For a faithfully indexed sequence
$A=\{a_n:n\in\N\}$ of elements of a topological group $G$,
the following statements are equivalent:
\begin{itemize}
\item[(i)] $A$ is $f$-Cauchy productive in $G$;
\item[(ii)]
 for every neighborhood $U$ of the identity of $G$ and every function $z:\N\to\Z$ with $|z|\le f$, there exists $k\in\N$ such that $\prod_{n=l}^m a_n^{z(n)}\in U$ for every $l,m\in\N$ satisfying $k\leq l\leq m$.
\end{itemize}
\end{lemma}

\begin{proof}
This follows from Lemma \ref{Cauchy:productive:criterion} and Definition \ref{def:ap-sequence}.
\end{proof}

For future reference, let us state explicitly a particular corollary of Lemma \ref{left:cauchy:condition}.

\begin{lemma}
\label{AP:is:null:sequence} Let
$f\in \Wi$.
Suppose that $\{a_n:n\in\N\}$ is an $f$-Cauchy productive sequence in a topological group $G$. Then the sequence
$\{a_n^{z(n)}:n\in\N\}$ converges to the identity $e$ of $G$ for every function $z:\N\to\Z$ satisfying $|z|\le f$;
in particular, $\{a_n:n\in\N\}$ converges to $e$.
\end{lemma}

Recall that a topological group $G$ is called {\em NSS\/} provided that there exists a
neighborhood of the identity of $G$ that contains no nontrivial subgroup of $G$.
Our next theorem strengthens \cite[Theorem 4.9]{Sp}.

\begin{theorem}\label{NSS:is:NACP} An NSS group contains no $f_\omega$-Cauchy productive sequences.
\end{theorem}

\begin{proof}
Let $A=\{a_n:n\in\N\}$ be a faithfully indexed sequence in an NSS group $G$. Fix a neighborhood $U$ of the identity $e$ of $G$
that contains no non-trivial subgroups of $G$. Then for each $n\in\N$  such that $a_n\neq e$
there exists $z_n\in\Z$  with $a_n^{z_n}\not\in U$. Clearly, the sequence $\{a_n^{z_n}:n\in\N\}$ does not converge to $e$.
Therefore, $A$ is not $f_\omega$-Cauchy productive by Lemma \ref{AP:is:null:sequence}.
\end{proof}

\begin{lemma}\label{if:f<g}
Let
$f,g \in \Wi$
be functions such that $g\le^*f$ (in particular; $g\le f$).
Then every $f$-productive ($f$-Cauchy productive) sequence
is also $g$-productive ($g$-Cauchy productive, respectively).
\end{lemma}

Our next proposition demonstrates that the notions of an $f^\star$-productive sequence and an $f^\star$-Cauchy
productive sequence may lead to something new only in non-commutative groups.

\begin{proposition}\label{fstar:is:f:for:abel} Let  $f \in \Wi$.
A sequence  $\{a_n:n\in\N\}$ of elements of an abelian topological group is $f$-summable ($f$-Cauchy summable) if and only if it is $f^\star$-summable ($f^\star$-Cauchy summable, respectively).
\end{proposition}

\begin{proof}
The ``only if'' part is obvious. To show the ``if'' part, suppose that the sequence $\{a_n:n\in\N\}$ is $f^\star$-summable ($f^\star$-Cauchy summable).
Fix a function $z:\N\to\Z$ such that $|z|\le f$. Define functions $z_+:\N\to\N$ and $z_-:\N\to\N$ by $z_+(n)=\max\{0,z(n)\}$ and  $z_-(n)=-\min\{0,z(n)\}$  for each $n\in\N$.
By our assumption, the sequences $\{z_+(n)a_n:n\in\N\}$ and $\{z_-(n)a_n:n\in\N\}$ are (Cauchy) summable. Then $\{-z_-(n)a_n:n\in\N\}$
is (Cauchy) summable as well. Since
$z(n)=z_-(n)+(-z_+(n))$  for every $n\in\N$, from Lemma \ref{sum:of:lin:is:lim:of:sum}
 we conclude that the sequence $\{z(n)a_n:n\in\N\}$ is (Cauchy) summable.
Therefore, the sequence $\{a_n:n\in\N\}$ is $f$-summable (respectively, $f$-Cauchy summable).
\end{proof}

\begin{proposition}
\label{bounded:f}
If $f\in \W$ is a bounded function, then a sequence of elements of an abelian topological group is $f$-summable ($f$-Cauchy summable) if and only if it is
$f_1^\star$-summable ($f_1^\star$-Cauchy summable, respectively).
\end{proposition}

\begin{proof}
Since $f_1\le f$, the ``only if'' part follows from Lemma \ref{if:f<g} and
Proposition \ref{fstar:is:f:for:abel}.
To prove the ``if'' part, assume that a sequence $A=\{a_n:n\in\N\}$ of points in an abelian topological group is $f_1^\star$-summable ($f_1^\star$-Cauchy summable).
By Proposition \ref{fstar:is:f:for:abel}, it suffices to show that $A$ is $f^\star$-summable ($f^\star$-Cauchy summable). To this end, we are going to prove that
the sequence $B=\{z(n)a_n:n\in\N\}$ is (Cauchy) summable whenever $z:\N\to\N$ is a function such that $z\le f$. Since $f$ is bounded, there exists $k\in\N$ such that $z(n)\le f(n)\le k$ for all $n\in\N$. For $i=1,\ldots,k$ define a function $z_i:\N\to\{0,1\}$ by letting $z_{i}(n)=0$ if $i> z(n)$ and $z_{i}(n)=1$ if $i\le z(n)$, for all $n\in\N$.
Note that $z=\sum_{i=1}^k z_i$, and so $z(n)a_n=\sum_{i=1}^k z_i(n) a_n$ for every $n\in\N$.
Since $A$ is $f_1^\star$-summable ($f_1^\star$-Cauchy summable), each sequence $\{z_i(n) a_n:n\in\N\}$ is (Cauchy) summable
for $i=1,\dots,k$. From Lemma \ref{sum:of:lin:is:lim:of:sum} (applied $k-1$ many times consequently), we conclude
that the sequence $B$ is (Cauchy) summable as well.
\end{proof}

\begin{remark}
\label{DomTar}
\begin{itemize}
\item[(i)]
Our terminology differs from that of Dominguez and Tarieladze \cite{DT-private}. Indeed, our $f_1^\star$-productive sequences are called super-multipliable sequences in \cite{DT-private}, and our $f_\omega$-productive sequences are precisely the hyper-multipliable  sequences
from \cite{DT-private}.
\item[(ii)]
A faithfully indexed sequence $\{a_n:n\in\N\}$ is $f_1^\star$-productive if and only if the series $\sum_{n=0}^\infty a_n$ is subseries convergent in the sense of \cite{Kalton}. 
\end{itemize}
\end{remark}

\section{Unconditionally $f$-(Cauchy) productive sequences, for a given weight function $f \in \Wi$}
\label{f-(Cauchy) productive sets, for a given function f: N to omega+1}

\begin{definition}
\label{f:productive:sets}
Let $f\in\Wi$.
\begin{itemize}
\item[(i)]
A faithfully indexed sequence $A=\{a_n:n\in\N\}$ in a topological group $G$
will be called {\em an unconditionally $f$-Cauchy productive\/} (an {\em unconditionally $f$-productive\/}) provided that the sequence
$\{a_{\varphi(n)}:n\in\N\}$ is $(f\circ\varphi)$-Cauchy productive (respectively, $(f\circ\varphi)$-productive) for every bijection $\varphi:\N\to\N$.
\item[(ii)]
When $G$ is abelian, we use the word ``summable'' instead of ``productive'', thereby obtaining
the abelian counterparts of an {\em unconditionally $f$-Cauchy summable sequence\/} and an {\em unconditionally $f$-summable sequence\/}.
\end{itemize}
\end{definition}

\begin{lemma}
\label{easy:connections:Cauchy}
Let $f\in \Wi$ and let $A=\{a_n:n\in\N\}$ be a faithfully indexed sequence of elements of a topological group.
\begin{itemize}
\item[(i)]
If $A$ is unconditionally $f$-productive, then $A$ is unconditionally $f$-Cauchy productive.
\item[(ii)]
If one additionally assumes that $G$ is sequentially Weil complete, then $A$ is unconditionally $f$-productive if and only if it is unconditionally $f$-Cauchy productive.
\end{itemize}
\end{lemma}

\begin{proof} (i) is clear, and (ii) follows from Lemma \ref{easy:connections}(ii) and Definition \ref{f:productive:sets}.
\end{proof}

Our next lemma shows that in  Definition \ref{f:productive:sets} bijections can be replaced with injections.

\begin{lemma}
\label{AP:set:has:AP:subsets}
For $f\in \Wi$ and a faithfully indexed sequence $A=\{a_n:n\in\N\}$ of elements of a topological group, the following conditions are equivalent:
\begin{itemize}
\item[(i)] $A$ is unconditionally $f$-Cauchy productive (unconditionally $f$-productive);
\item[(ii)] for every injection $\varphi:\N\to\N$, the sequence $\{a_{\varphi(n)}:n\in\N\}$ is  $(f\circ\varphi)$-Cauchy productive  (respectively, $(f\circ\varphi)$-productive).
\end{itemize}
\end{lemma}

\begin{proof} The implication (ii)$\to$(i) is obvious. Let us prove the implication (i)$\to$(ii). Assume that $A$ is an unconditionally $f$-Cauchy productive (unconditionally $f$-productive)  sequence in a topological group $G$, and consider an injection $\varphi:\N\to\N$ and a function $z:\N\to\Z$  such that
\begin{equation}\label{eq:|z|<f:varphi}
\mbox{$|z|\leq f\circ \varphi$.}
\end{equation}
We need to show that the sequence
$$
B=\left\{\prod_{j=0}^ka_{\varphi(j)}^{z(j)}:k\in\N\right\}
$$
 is left Cauchy (respectively,  converges in $G$).

Since $\varphi:\N\to\N$ is an injection, we can fix a bijection $\phi:\N\to\N$ and an order-preserving map $\sigma:\N\to\N$ such that
\begin{equation}
\label{eq:xi}
\varphi=\phi\circ\sigma,
\ \ \mbox{and so}\ \
\sigma=\phi^{-1}\circ \varphi.
\end{equation}
Define a map $z': \N \to \Z$ by
$$
z'(i)=\left\{\begin{array}{ll}
z\circ \sigma^{-1}(i) & \mbox{if $i\in\sigma(\N)$}\\
0 & \mbox{if $i\in\N\setminus \sigma(\N)$}
\end{array}
\right.
\ \ \
\mbox{for all}
\ \
i\in\N.
$$
In particular,
\begin{equation}
\label{z:and:z'}
z=z'\circ\sigma.
\end{equation}

Since $\sigma$ is order preserving and $a_{\phi(i)}^{z'(i)}=a_{\phi(i)}^0=e$ for $i\in\N\setminus\sigma(\N)$, from \eqref{eq:xi} and \eqref{z:and:z'} we get
\begin{equation}
\label{subsequence:product}
\prod_{i=0}^{\sigma(k)}a_{\phi(i)}^{z'(i)}
=
\prod_{j=0}^{k}
a_{\phi(\sigma(j))}^{z'(\sigma(j))}
=
\prod_{j=0}^{k}
a_{\phi\circ\sigma(j)}^{z'\circ\sigma(j)}
=
\prod_{j=0}^{k}
a_{\varphi(j)}^{z(j)}
\ \
\mbox{for every}
\ \
k\in\N.
\end{equation}

For $i\in \sigma(\N)$, one has
$
|z'(i)|=|z\circ \sigma^{-1}(i)|\le f\circ \varphi\circ\sigma^{-1}(i)= f\circ \varphi\circ\varphi^{-1}\circ\phi(i) =f\circ \phi(i)
$
by \eqref{eq:|z|<f:varphi} and \eqref{eq:xi}. For $i\in \N\setminus\sigma(\N)$, one again has
$|z'(i)|=0\le f\circ \phi(i)$. This shows that
\begin{equation}
\label{hoshi}
|z'| \leq f\circ\phi.
\end{equation}

Since $\phi$ is a bijection and $A$ is unconditionally $f$-Cauchy productive (unconditionally $f$-productive), the sequence
$\{a_{\phi(n)}:n\in\N\}$ is $(f\circ\phi)$-Cauchy productive (respectively, $(f\circ\phi)$-productive). Now from \eqref{hoshi} we conclude that the sequence
$$
C=\left\{\prod_{i=0}^ka_{\phi(i)}^{z'(i)}:k\in\N\right\}
$$
is left Cauchy (converges in $G$, respectively). Since $\sigma$ is order preserving, it follows from \eqref{subsequence:product} that $B$ is a subsequence of the sequence $C$.
Therefore, $B$ is also left Cauchy (converges in $G$, respectively).
\end{proof}

\begin{remark}
\label{reformulation:of:f-productivity}
\begin{itemize}
\item[(i)]
For a constant function $f\in \Wi$ (in particular, for $f_\omega$), Lemma \ref{AP:set:has:AP:subsets} says that a faithfully indexed sequence is unconditionally
$f$-Cauchy productive (unconditionally $f$-productive) if and only if all of its faithfully indexed subsequences are $f$-Cauchy productive ($f$-productive).
\item[(ii)]
 From Lemma \ref{AP:set:has:AP:subsets} it follows that unconditionally $f_\omega$-productive sequences coincide with countably
infinite absolutely productive sets in the sense of \cite[Definition 4.1]{Sp}.
\end{itemize}
\end{remark}

Item (ii) of this remark shows that our next definition is equivalent to \cite[Definition 4.5]{Sp}.

\begin{definition}
\label{def:TAP}
{\rm (\cite{Sp})} A topological group is {\em TAP\/} if it contains no unconditionally $f_\omega$-productive sequences.
\end{definition}

The following lemma is an unconditional analogue of Lemma \ref{if:f<g}.
\begin{lemma}
\label{unconditional:if:f<g}
Let $f,g \in \Wi$ be functions such that $g\le^*f$ (in particular, $g\le f$).
Then every unconditionally $f$-productive (unconditionally $f$-Cauchy productive) sequence is also unconditionally
$g$-productive (unconditionally $g$-Cauchy productive, respectively).
\end{lemma}

\begin{proof}
Let $A=\{a_n:n\in\N\}$ be an unconditionally $f$-productive (unconditionally $f$-Cauchy productive) sequence. Let $\varphi:\N\to\N$ be an arbitrary bijection. Then
the sequence $B=\{a_{\varphi(n)}:n\in\N\}$ is $(f\circ\varphi)$-productive (respectively, $(f\circ\varphi)$-Cauchy productive). Since $g\le^*f$ and $\varphi$ is a bijection,
$g\circ\varphi\le^* f\circ\varphi$. From Lemma \ref{if:f<g}, we conclude that the sequence $B$ is $(g\circ\varphi)$-productive (respectively,
$(g\circ\varphi)$-Cauchy productive). Since $\varphi:\N\to\N$ was an arbitrary bijection,
it follows that $A$ is unconditionally $g$-productive (unconditionally $g$-Cauchy productive).
\end{proof}

The following lemma is a reformulation of the first part of \cite[Lemma 2.4(a)]{DT-private} in our terminology; see Remark~\ref{DomTar}(i).

\begin{lemma}
{\rm (\cite[Lemma 2.4(a)]{DT-private})}
 \label{vajas:lemma}
If $\{b_n:n\in\N\}$ is an $f_1^\star$-summable sequence in an abelian topological
group, then the sequence $\{b_{\varphi(n)}:n\in\N\}$ is summable for every bijection $\varphi:\N\to\N$.
\end{lemma}

\begin{corollary}
\label{abelian:corollary}
For $f\in\Wi$, a sequence in an abelian topological group is $f$-summable ($f$-Cauchy summable)
if and only if it is unconditionally $f$-summable (unconditionally $f$-Cauchy summable, respectively).
\end{corollary}

\begin{proof}
We first prove the ``$f$-summable'' version. The ``if'' part follows from Definition \ref{f:productive:sets}. To prove the ``only if'' part, assume that $A=\{a_n:n\in\N\}$ is an $f$-summable sequence in an abelian topological group $G$. Let $\varphi:\N\to\N$ be a bijection and let $z:\N\to\Z$ be a function such that $|z|\le f\circ\varphi$. We have to show that the sequence
$\{z(n)a_{\varphi(n)}:n\in\N\}$ is summable in $G$.

For every $n\in\N$, let $b_n=z(\varphi^{-1}(n))a_n$. For each map $\varepsilon:\N\to\{0,1\}$, we have $|\varepsilon(n) \cdot z(\varphi^{-1}(n))|\le |z\circ\varphi^{-1}(n)|\leq f(n)$
for all $n\in\N$, and since $A$ is an $f$-summable sequence, the sequence $\{\varepsilon(n)\cdot (z(\varphi^{-1}(n))a_n:n\in\N\}=\{\varepsilon(n) b_n:n\in\N\}$
must be summable. It follows that the sequence $\{b_n:n\in\N\}$ is $f_1^\star$-summable. By Lemma \ref{vajas:lemma}, the sequence
$\{b_{\varphi(n)}:n\in\N\}$ is summable. Since $b_{\varphi(n)}=z(\varphi^{-1}(\varphi(n)))a_{\varphi(n)}=z(n) a_{\varphi(n)}$ for all $n\in\N$, this finishes the proof.

Next, we prove the ``$f$-Cauchy summable'' version. The ``if'' part follows from Definition \ref{f:productive:sets}(ii).
To prove the ``only if'' part, assume that $A=\{a_n:n\in\N\}$ is an $f$-Cauchy summable sequence in an abelian
topological group $G$. The completion $H$ of $G$ is a sequentially Weil complete abelian group. Clearly, $A$ is an $f$-Cauchy summable sequence in $H$, so
$A$ is $f$-summable in $H$ by Lemma \ref{easy:connections}(ii). By the first part of the proof, $A$ is unconditionally $f$-summable in $H$.
Applying Lemma \ref{easy:connections:Cauchy}(i), we conclude that $A$ is unconditionally $f$-Cauchy summable in $H$ (and thus, also in $G$).
\end{proof}

\begin{remark}
In Theorem \ref{TAP:group:with:ap:sequence} we construct an $f_\omega$-productive sequence $B=\{b_n:n\in\N\}$ in a (necessarily non-abelian) separable metric group $H$ and a bijection $\varphi:\N\to\N$ such  that the sequence $\{b_{\varphi(n)}:n\in\N\}$ is not productive in $H$. Since $B$ is clearly $f_1^\star$-productive,
this shows that Lemma \ref{vajas:lemma} fails for non-commutative groups, thereby providing a strong negative answer to a question of Dominguez and Tarieladze from
\cite[Remark 2.5]{DT-private}. Note that, for every function $f\in\Wi$, the sequence $B$ is $f$-productive but not unconditionally $f$-productive in $H$. This demonstrates that
Corollary \ref{abelian:corollary} strongly fails for non-commutative groups as well.
\end{remark}

\section{Unconditionally $f$-(Cauchy) productive sequences in metric groups}
\label{f-(Cauchy) productive sets in metric groups}

\begin{lemma}\label{reshuffling}
Let $\{U_n:n\in\N\}$ be a family of subsets of a group $G$ such that
$e\in U_{n+1}^3\subseteq U_{n}$ for every integer $n\in\N$. Assume
that $\varphi:\N\to\N$ is an injection, $l$ and $m$
are integers
satisfying
$0\leq l\leq m$ and
$s=\min\{\varphi(n):n\in\N, l\le n\le m\}>0$.
Then
$\prod_{n=l}^m U_{\varphi(n)}\subseteq U_{s-1}$.
\end{lemma}

\begin{proof} We will prove our lemma by induction on $m-l$.
If $m-l=0$, then $\prod_{n=l}^m U_{\varphi(n)}=U_{\varphi(l)}=U_s\subseteq
U_{s-1}$. (Note that $U_{s-1}$ is defined since $s>0$.)

Assume that $m-l\ge 1$, and suppose that the conclusion of our lemma holds for all integers $l',m'$ such that $0\leq l'\leq m'$ and $m'-l'<m-l$. Choose an integer $i$ such that $l\le i\le m$ and $\varphi(i)=\min\{\varphi(n):n=l,\ldots,m\}=s$. Since $\varphi$ is an injection,
$\min\{\varphi(n):n=l,\ldots,i-1\}> s$ (if $i>l$) and $\min\{\varphi(n):n=i+1,\ldots,m\}> s$ (if $i<m$).
Applying the inductive assumption, we obtain
$$
\prod_{n=l}^m
U_{\varphi(n)}=\left(\prod_{n=l}^{i-1}U_{\varphi(n)}\right)U_{\varphi(i)}\left(\prod_{n=i+1}^mU_{\varphi(n)}\right)\subseteq
U_{s}U_{s}U_{s}\subseteq U_{s-1},
$$
where the left or the right factor of the product may be equal to $e$ (when $i=l$ or $i=m$, respectively).
\end{proof}

Our next lemma offers a way to build unconditionally $f$-Cauchy  productive sequences in a metric group.

\begin{lemma}
\label{building:f-Cauchy:sequences:in:a:metric:group} Let $f\in \Wi$ and let $\mathscr{U}=\{U_n:n\in\N\}$ be an open base at $e$ of a metric group $G$ such that
$U_{n+1}^3\subseteq
U_n$ for every $n\in\N$. Assume that
$A=\{a_n:n\in\N\}$
is a
faithfully indexed sequence of elements of $G$ such that
\begin{equation}
\label{bounded:multiples:in:U} \mbox{$\{a_n^z: z\in\Z, |z|\le
f(n)\}\subseteq U_n$ for every $n\in\N$.}
\end{equation}
 Then
$A$
is unconditionally $f$-Cauchy productive.
\end{lemma}

\begin{proof}
According to Definition \ref{f:productive:sets}, we must prove that
the sequence $\{a_{\varphi(n)}:n\in\N\}$ is $(f\circ\varphi)$-Cauchy
productive for every bijection $\varphi:\N\to\N$. Fix such a
bijection $\varphi$. Let $z:\N\to\Z$ be a function such that $|z|\le
f\circ\varphi$. Note that $|z(n)|\le f(\varphi(n))$ for every
$n\in\N$, so \eqref{bounded:multiples:in:U} yields
\begin{equation}
\label{b_n:equation} b_n= a_{\varphi(n)}^{z(n)}\subseteq
U_{\varphi(n)} \mbox{ for every } n\in\N.
\end{equation}
It remains only to show that the sequence $\{b_{n}:n\in\N\}$ is
Cauchy productive. Let $U$ be an open neighborhood of $e$ in $G$.
Since $\mathscr{U}$ is a base at $e$, there exists
$i\in\N\setminus\{0\}$ such that $U_{i}\subseteq U$. Define
$k=\max\{\varphi^{-1}(n):n=0,1,\dots, i\}+1$. Suppose that
$l,m\in\N$ and $k\le l\le m$. Then $s=\min\{\varphi(n):n\in\N, l\le
n\le m\}>i$, and so
$$
\prod_{n=l}^m b_n \in \prod_{n=l}^mU_{\varphi(n)}\subseteq
U_{s-1}\subseteq U_i \subseteq U
$$
by \eqref{b_n:equation} and Lemma \ref{reshuffling}. Applying the
implication (ii)$\to$(i) of Lemma \ref{Cauchy:productive:criterion},
we conclude that the sequence $\{b_n:n\in\N\}$ is Cauchy productive.
\end{proof}

\begin{theorem}\label{non1-TAP}
A non-discrete metric group contains an unconditionally $f$-Cauchy productive sequence for every $f\in  \W$.
\end{theorem}

\begin{proof}
Let $G$ be a non-discrete metric group, and let $\{U_n:n\in\N\}$ be
as in the assumption of Lemma
\ref{building:f-Cauchy:sequences:in:a:metric:group}. For every
$n\in\N$, the set
$$
\mbox{$V_n=\{x\in G: x^k\in U_n$ for all $k\in\Z$ with $|k|\le
f(n)$\}}
$$
 is open in $G$.  Since $G$ is non-discrete, by induction on $n\in\N$ we can choose $a_n\in V_n\setminus\{a_0,\dots,a_{n-1}\}$. Clearly,
 the sequence
$A=\{a_n:n\in\N\}$ is faithfully indexed. Applying Lemma
\ref{building:f-Cauchy:sequences:in:a:metric:group}, we conclude
that $A$ is unconditionally $f$-Cauchy productive.
\end{proof}

\begin{corollary}\label{no:f-product:seq:in:nondiscrete:metr:group}
A non-discrete metric Weil complete  group contains an unconditionally $f$-productive sequence for every $f\in  \W$.
\end{corollary}

The real line $\mathbb{R}$ is NSS, so it does not contain $f_\omega$-Cauchy productive sequences by Theorem \ref{NSS:is:NACP}.
This shows that one cannot replace $f\in\W$ with $f_\omega$ either in Theorem \ref{non1-TAP} or Corollary \ref{no:f-product:seq:in:nondiscrete:metr:group}.
Moreover, the non-existence of $f_\omega$-Cauchy productive sequences characterizes the NSS property in metric groups.

\begin{theorem}\label{metric:NACP:iff:NSS}
For a metric group $G$ the following conditions are equivalent:
\begin{itemize}
\item[(i)] $G$ is NSS;
\item[(ii)] $G$ does not contain an $f_\omega$-Cauchy productive sequence;
\item[(iii)] $G$ does not contain an
unconditionally $f_\omega$-Cauchy productive sequence.
\end{itemize}
\end{theorem}

\begin{proof}
The implication (i)$\to$(ii) is proved in Theorem \ref{NSS:is:NACP}. The implication (ii)$\to$(iii) is clear.

It remains to show that (iii)$\to$(i). Assume that $G$ is not NSS. Let $\{U_n:n\in\N\}$ be as in the assumption of Lemma~\ref{building:f-Cauchy:sequences:in:a:metric:group}. Since $G$ is not NSS, for every $n\in\N$ there exists $a_n\in G\setminus\{0\}$ such that $\{a_n^z:z\in\Z\}\subseteq U_n$. Since the sequence
$\{U_n:n\in\N\}$ is decreasing, we may assume, without loss of generality, that $a_n\not\in \{a_0,\dots,a_{n-1}\}$ for every
$n\in\N$. Thus, $A=\{a_n:n\in\N\}$ is a faithfully indexed sequence satisfying the assumption of Lemma \ref{building:f-Cauchy:sequences:in:a:metric:group} with
$f=f_\omega$. Applying this lemma, we conclude that $A$ is unconditionally $f_\omega$-Cauchy productive.
\end{proof}

Combining Theorem \ref{metric:NACP:iff:NSS} and Lemma
\ref{productive:and:Cauchy:productive:coincide}, we obtain the following

\begin{corollary}\label{weil:NSS:iff:TAP}
For a  Weil complete metric group $G$ the following conditions are equivalent:
\begin{itemize}
\item[(i)] $G$ is NSS;
\item[(ii)] $G$ does not contain an $f_\omega$-productive sequence;
\item[(iii)] $G$ does not contain an unconditionally $f_\omega$-productive sequence (that is, $G$ is TAP).
\end{itemize}
\end{corollary}

The implication (ii)$\to$(i) of this corollary is proved also in \cite{DT}, and the implication (iii)$\to$(i) is proved also in \cite{DT-private}.

The assumption that $G$ is metrizable is essential in both Theorem \ref{metric:NACP:iff:NSS} and
Corollary \ref{weil:NSS:iff:TAP}; see Example \ref{ex:not:NSS:NACP}.

\section{Unconditionally $f$-(Cauchy) productive sequences in linear groups}
\label{f-(Cauchy) productive sets in linear groups}

 Call a topological group $G$  {\em  linear\/} (and its topology a {\em linear group topology\/})
 if $G$ has a base of  $e$ formed by open subgroups of $G$. Clearly, a linear group is NSS if and only if it is discrete.
 Therefore,  the non-discrete linear groups can be considered as strongly missing the NSS property.

Lemma \ref{left:cauchy:condition} can be simplified substantially for a linear group $G$.

\begin{theorem}\label{AP-Cauchy:seq:in:lin:groups}
For a faithfully indexed sequence $A=\{a_n:n\in\N\}$ of points of a linear group $G$, the following statements are equivalent:
\begin{itemize}
\item[(i)] $A$ converges to $e$;
\item[(ii)] $A$ is $f$-Cauchy productive for some $f\in \Wi$;
\item[(iii)] $A$ is $f$-Cauchy productive for every $f\in \Wi$;
\item[(iv)] $A$ is unconditionally $f_\omega$-Cauchy productive.
\end{itemize}
\end{theorem}

\begin{proof}
(i)$\to$(iv)
Take an arbitrary bijection $\varphi:\N\to\N$. To establish (iv), it suffices to show that the sequence
$\{a_{\varphi(n)}:n\in\N\}$ is $(f_\omega\circ\varphi)$-Cauchy productive. Since $f_\omega=f_\omega\circ\varphi$,
we must prove that, for an arbitrary function $z:\N\to\Z$,
the sequence $\left\{a_{\varphi(n)}^{z(n)}:n\in\N\right\}$ is Cauchy productive. Let $U$ be a neighborhood of the identity of $G$.
Since $G$ is linear, there exists an open subgroup $H$ of $G$ with $H\subseteq U$. Since
$\{a_n:n\in\N\}$
converges to $e$ by (i), so does the sequence $\{a_{\varphi(n)}:n\in\N\}$. Therefore, there is some $k\in\N$ such
that $a_{\varphi(n)}\in H$ for every integer $n\geq k$. Since $H$ is a group, it follows that $\prod_{n=l}^m a_{\varphi(n)}^{z(n)}\in H\subseteq U$
for every $l,m\in\N$
such that $k\leq l\le m$.  Therefore, the sequence $\left\{a_{\varphi(n)}^{z(n)}:n\in\N\right\}$ is Cauchy productive by Lemma
\ref{Cauchy:productive:criterion}.

(iv)$\to$(iii) and (iii)$\to$(ii) are straightforward, and (ii)$\to$(i) follows from Lemma \ref{AP:is:null:sequence}.
\end{proof}

\begin{corollary}
\label{Weil:complete:linear:group}
For a faithfully indexed sequence $A=\{a_n:n\in\N\}$ of points of a sequentially Weil complete linear group,
the following statements are equivalent:
\begin{itemize}
\item[(i)] $A$ converges to $e$;
\item[(ii)] $A$ is $f$-productive for some $f\in \Wi$;
\item[(iii)] $A$ is $f$-productive for every $f\in \Wi$;
\item[(iv)] $A$ is unconditionally $f_\omega$-productive.
\end{itemize}
\end{corollary}

\begin{corollary}\label{NACP:iff:no:conv:sequences}\label{char:of:complete:linear:TAP:groups}
A linear (sequentially Weil complete) group contains an unconditionally $f_\omega$-Cauchy productive sequence (an unconditionally $f_\omega$-productive sequence)
if and only if it contains a nontrivial convergent sequence.
\end{corollary}

\begin{remark}
Let $G$ be the symmetric group discussed in Example \ref{symmetric:group:example}. Note that {\em $G$ has a liner metric topology\/}.
This example shows that one cannot replace ``sequentially Weil complete'' by ``Raikov complete'' in Corollary \ref{Weil:complete:linear:group}.
\end{remark}

Our next two corollaries show that Theorem \ref{non1-TAP}  and Corollary \ref{no:f-product:seq:in:nondiscrete:metr:group}
can be significantly strengthened for linear groups. (Note that every complete group is trivially sequentially complete.)

\begin{corollary}
\label{sequentially:complete:liner:TAP:group:is:discrete} A non-discrete  (sequentially Weil complete) sequential linear group
contains an unconditionally $f_\omega$-Cauchy productive sequence (an unconditionally $f_\omega$-productive sequence).
\end{corollary}

\begin{proof}
A non-discrete sequential group contains a nontrivial convergent sequence, and we can apply Corollary \ref{char:of:complete:linear:TAP:groups}.
\end{proof}

\begin{corollary}
\label{linear:complete:corollary}
A non-discrete (Weil complete) metric linear group contains an unconditionally $f_\omega$-Cauchy productive sequence
(an unconditionally $f_\omega$-productive sequence).
\end{corollary}

\begin{corollary} For a linear sequentially Weil complete sequential group $G$, the following statements are equivalent:
\begin{itemize}
\item[(i)] $G$ is TAP;
\item[(ii)] $G$ is discrete.
\end{itemize}
\end{corollary}

\begin{theorem}\label{linear:TAP}
For a linear group $G$, the following statements are equivalent:
\begin{itemize}
\item[(i)] $G$ is sequentially  Weil complete and does not contain an $f_\omega$-productive sequence;
\item[(ii)] $G$ is sequentially  Weil complete and does not contain an unconditionally $f_\omega$-productive sequence;
\item[(iii)] $G$ has no nontrivial convergent sequences.
\end{itemize}
\end{theorem}
\begin{proof}
(i)$\to$(ii) holds since every unconditionally $f_\omega$-productive sequence is $f_\omega$-productive.

(ii) $\Rightarrow$(iii) follows from Corollary \ref{char:of:complete:linear:TAP:groups}.

(iii)$\to$(i) follows from Lemma
\ref{G:with:no:conv:seq:is:seq:Weil:complete} and
Corollary
\ref{char:of:complete:linear:TAP:groups}.
\end{proof}

Let us give an example of a topological group satisfying three equivalent conditions of the above theorem.

\begin{example}\label{ex:B:with:cocount:top}\label{ex:not:NSS:NACP}  Let $B$ be the Boolean group of size $\cont$  equipped with the
co-countable topology having as a base of neighborhoods of $0$ all subgroups of at most countable index. Let $G$ be the completion of $B$.
\begin{itemize}
\item[(i)] {\em $G$ is a complete non-discrete linear group without nontrivial convergent sequences\/} because every $G_\delta$-subset of $G$ is open in $G$.
\item[(ii)] {\em $G$ is not NSS and contains no (unconditionally) $f_\omega$-Cauchy productive sequence\/}.
Indeed, $G$ is not not NSS because it is non-discrete and linear, and $G$ does not contain an (unconditionally) $f_\omega$-Cauchy productive sequence by Corollary \ref{NACP:iff:no:conv:sequences}.
\end{itemize}
\end{example}

Example \ref{ex:not:NSS:NACP} shows that metrizability cannot be dropped from the assumptions of Theorem \ref{metric:NACP:iff:NSS} and  Corollary \ref{weil:NSS:iff:TAP}.

\section{Unconditionally $f$-productive sequences in locally compact and $\omega$-bounded groups}
\label{f-productive sets in locally compact and omega-bounded groups}

\begin{proposition}
\label{totally:disconnected:locally:compact:have:f_omega:sets}
A non-discrete locally compact totally disconnected group contains an unconditionally $f_\omega$-productive sequence.
\end{proposition}
\begin{proof}
It is well known that every non-discrete locally compact group contains a non-trivial convergent sequence; see, for example,  \cite{Sha}. Moreover, totally disconnected locally compact groups have a linear topology \cite[Theorem (7.7)]{HR}.  The conclusion now follows from Corollary \ref{char:of:complete:linear:TAP:groups}.
\end{proof}

Recall that a topological group $G$ is {\em $\omega$-bounded\/} if the closure of every countable subset of $G$ is compact.

\begin{corollary}
\label{tot:disc:Obounded:TAP:iff:finite}
An infinite totally disconnected $\omega$-bounded group contains an unconditionally $f_\omega$-productive sequence.
\end{corollary}

\begin{proof}
Let $D$ be a countably infinite subset of a totally disconnected $\omega$-bounded group $G$. Then the subgroup $H$ of $G$ generated by $D$ is countable. Since $G$ is $\omega$-bounded, 
the closure $K$ of $H$ in $G$
is a compact subgroup of $G$. Since $K$ is infinite, it is non-discrete. As a subgroup of a totally disconnected group $G$, $K$ is totally disconnected.
 From Proposition \ref{totally:disconnected:locally:compact:have:f_omega:sets} we conclude that $K$ contains an unconditionally $f_\omega$-productive sequence.
\end{proof}

The following example shows that neither local compactness nor $\omega$-boundedness can be omitted in Proposition
\ref{totally:disconnected:locally:compact:have:f_omega:sets} and Corollary \ref{tot:disc:Obounded:TAP:iff:finite}.

\begin{example}\label{last:example} The metric  zero-dimensional (hence, totally disconnected) group $\Q$ does not contain any $f_\omega$-Cauchy productive sequence. Indeed, $\Q$ is a subgroup of $\R$ that does not contain any $f_\omega$-Cauchy productive sequence.
\end{example}

\begin{theorem}\label{zerodim:TAP:compact:is:finite}
A non-discrete locally compact group contains an unconditionally $f$-productive sequence for every function $f\in \W$.
\end{theorem}

\begin{proof}
If $G$ is metrizable, then  the conclusion follows from  Corollary \ref{no:f-product:seq:in:nondiscrete:metr:group},
since locally compact groups are Weil complete. Suppose now that $G$ is not metrizable. By  theorem of Davis \cite{Dav}, $G$ is
homeomorphic to a product $K \times \R^n \times D$, where $K$ is a compact subgroup of $G$, $n\in \N$ is
an integer and $D$ is a discrete space. Our assumption on $G$ implies that the compact subgroup $K$ of $G$ is not metrizable. We consider two cases now.

\medskip
{\it Case 1\/}: {\sl $K$ is torsion\/}. In this case  $K$ is totally disconnected \cite{D}, so $K$ contains an unconditionally $f_\omega$-productive sequence $A$ by Proposition \ref{totally:disconnected:locally:compact:have:f_omega:sets}. Clearly, $A$ is also unconditionally $f$-productive.

\medskip
{\it Case 2\/}: {\sl $K$ is not torsion\/}. Let $a$ be a non-torsion element of $K$. Then the closed subgroup $N$ of $K$ generated by $a$ is a compact abelian group.
By \cite[Theorem 2.5.2]{Rudin}, $N$ contains an infinite compact metrizable subgroup $M$. Since compact groups are Weil complete, applying  Corollary \ref{no:f-product:seq:in:nondiscrete:metr:group} to  $M$, we  find an unconditionally $f$-productive sequence in $M$ (and thus, in $G$ as well).
\end{proof}

\begin{corollary}
A locally compact group is discrete if and only if it does not contain an $f_1^\star$-productive sequence.
\end{corollary}

\begin{corollary}
\label{omega:bounded}
An infinite $\omega$-bounded (in particular, compact) group has an unconditionally $f$-productive sequence for every function $f\in \W$.
\end{corollary}

\begin{proof} Let $G$ be an infinite $\omega$-bounded group. Arguing as in the proof of Corollary \ref{tot:disc:Obounded:TAP:iff:finite},
we find an infinite compact subgroup $K$ of $G$. Now we can apply Theorem \ref{zerodim:TAP:compact:is:finite} to $K$.
\end{proof}

\section{Descriptive properties of groups having $f_1^\star$-productive sequences}
\label{Descriptive properties of groups having f_1star-productive sequences}

We will need the following lemma essentially due to Kalton. In
the abelian case this lemma is proved in
\cite{Kalton} (see also \cite{Jameson}),
and its non-commutative version appears
in \cite{Drewnowski}, with the proof attributed to Kalton.
\begin{lemma}
\label{lemma:Kalton}
Let $A=\{a_n:n\in\N\}$ be an $f_1^\star$-productive sequence in a
topological group $G$.
Then the map $\lambda_A:2^\N\to G$ defined by
$$
\lambda_A(f)=\lim_{n\to\infty} \prod_{i=0}^{n} a_{i}^{f(i)}
\ \
\mbox{for every}
\ \
f\in 2^\N,
$$
is continuous when $2^\N$ is taken with
its usual Cantor set topology.
\end{lemma}

Our next theorem provides an additional information about the ``limit map''
$\lambda$.
\begin{theorem}
\label{Cantor:set:theorem}
Every $f_1^\star$-productive sequence $A$ contains a subsequence $B$ such that $\lambda_B$ is a homeomorphic embedding.
\end{theorem}

\begin{proof}
Let $A=\{a_k:k\in\N\}$ be an $f_1^\star$-productive sequence in a topological group $G$.
For $X\subseteq G$
we denote by $\overline{X}$ the closure of $X$ in $G$. 

By induction on $n\in\N$ we choose an integer $k_n\in\N$ and an open neighborhood $U_n$ of $e$ with the following properties:
\begin{itemize}
\item[(i$_n$)] if $n\ge 1$, then $k_{n-1}<k_n$,
\item[(ii$_n$)] if $n\ge 1$, then  $a_{k_n}\in U_{n-1}$,
\item[(iii$_n$)] if $n\ge 1$, then  $U_n^2\subseteq U_{n-1}$,
\item[(iv$_n$)] $b_f\overline{U_n^2}\cap b_g\overline{U_n^2}=\emptyset$ whenever $f,g\in 2^{\{0,1,\dots,n\}}$ and $f\not=g$,
where $b_h=\prod_{i=0}^{n} a_{k_i}^{h(i)}$ for $h\in 2^{\{0,1,\dots,n\}}$.
\end{itemize}

For $n=0$ choose $k_0\in\N$ with $a_{k_0}\not = e$ and fix 
an
open neighborhood $U_0$ of $e$ satisfying
(iv$_0$). Note that (i$_0$), (ii$_0$) and (iii$_0$) are trivially satisfied for $n=0$.

Suppose now that an integer $k_m\in\N$ and an  open neighborhood $U_m$ of $e$ have already been constructed so that the
properties (i$_m$), (ii$_m$), (iii$_m$) and (iv$_m$) hold for every integer $m$ with $0\le m\le n$. Let us define an integer
$k_{n+1}\in\N$ and an  open neighborhood $U_{n+1}$ of $e$ satisfying properties (i$_{n+1}$), (ii$_{n+1}$), (iii$_{n+1}$) and (iv$_{n+1}$). It follows from
(iv$_n$) that
\begin{equation}
\label{eq:8.2}
e\not\in \{b_f^{-1} b_g: f,g\in 2^{\{0,1,\dots,n\}}, f\not=g\}.
\end{equation}
Since $\lim_{k\to\infty} a_k=e$ by Lemma \ref{AP:is:null:sequence}, we can  choose $k_{n+1}\in \N$ such that $k_{n}<k_{n+1}$ and
\begin{equation}
\label{eq:8.1}
e\not=a_{k_{n+1}}\in U_{n}\setminus
\{b_f^{-1} b_g: f,g\in 2^{\{0,1,\dots,n\}}, f\not=g\}.
\end{equation}
In particular, (i$_{n+1}$) and (ii$_{n+1}$) hold.

 From \eqref{eq:8.2} and \eqref{eq:8.1} one easily concludes that $b_f\not=b_g$ whenever $f,g\in 2^{\{0,1,\dots,n+1\}}$ and $f\not=g$. This allows us to fix an
open neighborhood $U_{n+1}$ of $e$ satisfying (iii$_{n+1}$) and (iv$_{n+1}$).
The inductive step is now complete.

Since (i$_{n}$) holds for every $n\in\N$, the sequence
$B=\{a_{k_n}:n\in\N\}$ is a subsequence of $A$. Since $A$ is $f_1^\star$-productive in $G$, so is $B$.
Let us show that $B$ is the required subsequence of $A$.

First, we claim that
\begin{equation}
\label{eq:lambda}
\lambda_B(f)\in
\bigcap\left\{ \left(\prod_{i=0}^{n} a_{k_i}^{f(i)}\right)\overline{U_n^2}:n\in\N\right\}
=
\bigcap\left\{b_{f\restriction_{\{0,1,\dots,n\}}}\overline{U_n^2}:n\in\N\right\}.
\end{equation}
To show this, fix $n\in\N$.
Since (ii$_j$) and  (iii$_j$) hold for all $j\in\N$,
for every $m\in\N$ with $m\ge n$, we have
$$
\prod_{i=0}^{m} a_{k_i}^{f(i)}
\in
\left(\prod_{i=0}^{n} a_{k_i}^{f(i)}\right) U_{n}U_{n+1}\dots U_{m-1}
\subseteq
\left(\prod_{i=0}^{n} a_{k_i}^{f(i)}\right)U_{n}^2,
$$
which yields
$$
\lambda_B(f)=
\lim_{m\to\infty} \prod_{i=0}^{m} a_{k_i}^{f(i)}
\subseteq
\overline{\left(\prod_{i=0}^{n} a_{k_i}^{f(i)}\right)U_{n}^2}
=
\left(\prod_{i=0}^{n} a_{k_i}^{f(i)}\right)\overline{U_{n}^2}.
$$
Since
$n\in\N$ was chosen arbitrarily,
this finishes the proof of \eqref{eq:lambda}.

Let $f,g\in 2^\N$ and $f\not=g$. Then there exists $n\in\N$ such that $f'=f\restriction_{\{0,1,\dots,n\}}\not=g\restriction_{\{0,1,\dots,n\}}=g'$.
Therefore, $b_{f'}\overline{U_n^2}\cap b_{g'}\overline{U_n^2}=\emptyset$ by (iv$_n$). Since $\lambda_B(f)\in b_{f'}\overline{U_n^2}$ and $\lambda_B(g)\in
b_{g'}\overline{U_n^2}$, we conclude that $\lambda_B(f)\not=\lambda_B(g)$. This shows that $\lambda_B:2^\N\to G$ is an injection. Furthermore, $\lambda_B$ is continuous by Lemma \ref{lemma:Kalton}. As a continuous injection from a compact space to a Hausdorff space, $\lambda_B$ is a homeomorphic embedding.
\end{proof}

\begin{corollary}
\label{size:of:f_1-productive:groups>c} A topological group  having an $f_1^\star$-productive sequence contains a (subspace homeomorphic to the) Cantor set.
\end{corollary}

\begin{corollary}\label{|G|<cont:is:NAP}
A topological group of size $<\cont$ does not contain an $f$-productive sequence for any function $f\in \Wi$.
\end{corollary}

\begin{corollary}\label{|G|<cont:is:TAP}
Every topological group of size $<\cont$ is TAP.
\end{corollary}

\begin{remark} 
A careful analysis of  the proof of Theorem \ref{Cantor:set:theorem} shows that no recourse to Lemma \ref{lemma:Kalton} is necessary for proving that $|\lambda_A(2^\N)|\ge \cont$ for every $f_1^\star$-productive sequence $A$ in $G$. In particular, Lemma \ref{lemma:Kalton} is not necessary to prove
Corollaries \ref{|G|<cont:is:NAP} and \ref{|G|<cont:is:TAP}.
\end{remark}

\begin{example}\label{NACP:size<c}\label{question1}
Consider the topological group $G=(\Z,\tau_p)$, where $\tau_p$ is the $p$-adic topology on $\Z$. Then:
{\em
\begin{itemize}
\item[(i)]  $G=(\Z,\tau_p)$ is a countable precompact metric group with a linear topology;
\item[(ii)] $G$ is not NSS;
\item[(iii)] $G$ contains an unconditionally $f_\omega$-Cauchy productive sequence;
\item[(iv)] $G$ does not contain an $f_1^\star$-productive sequence, and so $G$ does not contain an $f$-productive sequence for any function $f\in \Wi$; in particular, $G$ is TAP.
\end{itemize}
}
Indeed, (i) is straightforward. Since $G$ is non-discrete and has a liner topology, we get (ii). Item (iii) follows from (i), (ii) and Theorem \ref{metric:NACP:iff:NSS}. Finally, (iv) follows from Corollary
\ref{size:of:f_1-productive:groups>c}, because $G$ is countable.
\end{example}

\begin{remark}
Example \ref{NACP:size<c} demonstrates that:
\begin{itemize}
\item[(i)]
sequential completeness cannot be omitted in Corollary \ref{sequentially:complete:liner:TAP:group:is:discrete},
\item[(ii)]
completeness cannot be dropped in Corollary \ref{linear:complete:corollary}.
\item[(iii)]
local compactness cannot be replaced by precompactness
in Theorem
\ref{zerodim:TAP:compact:is:finite},
\item[(iv)]
compactness cannot be weakened to precompactness
in Corollary \ref{omega:bounded},
and
\item[(v)]
``$f$-productive'' can not be replaced by ``$f$-Cauchy productive'' in Corollary
\ref{|G|<cont:is:NAP}.
\end{itemize}
\end{remark}

\begin{remark}
Combining Corollaries
 \ref{no:f-product:seq:in:nondiscrete:metr:group} and
\ref{size:of:f_1-productive:groups>c},
we obtain the well-known fact that every non-discrete  Weil complete metric group contains a homeomorphic copy of the Cantor set.
\end{remark}

\section{An application to function spaces $C_p(X,G)$: a solution of \cite[Question 11.1]{Sp}}
\label{sec:appl} As an application of our results, we offer a solution of a question from \cite{Sp}. Given a topological group $G$
and a space $X$, the symbol $C_p(X,G)$ denotes the topological subgroup of $G^X$ consisting of all continuous functions from $X$ to $G$.
In \cite{Sp} a space $X$ is called:
\begin{enumerate}
\item[(i)]{\em$G$-regular\/} if for every $x\in X$ and each closed set $F\subseteq X$ with $x\not\in F$ there exist
$f\in\Cp{X}{G}$ and $g\in G\setminus\{e\}$ such that $f(x)=g$ and $f(F)\subseteq\{e\}$;
\item[(ii)]{\em $G^\star$-regular\/} if there exists $g\in G\setminus\{e\}$ such that for every closed set $F\subseteq X$ and each $x\not\in F$ one can
find $f\in\Cp{X}{G}$ such that $f(x)=g$ and $f(F)\subseteq\{e\}$;
\item[(iii)]
{\em $G^{\star\star}$-regular} provided that, whenever $g\in G$, $x\in X$ and $F$ is a closed subset of $X$ satisfying $x\not\in F$,
 there exists $f\in\Cp{X}{G}$ such that $f(x)=g$ and $f(F)\subseteq\{e\}$.
\end{enumerate}

One of the main results from \cite{Sp} says that, for an NSS group $G$, a $G$-regular space $X$ is pseudocompact if and only if the group $C_p(X,G)$ is TAP (see \cite[Theorem 6.5]{Sp}). Demonstrating limits of this result, it was  proved in \cite{Sp} that there exist a precompact TAP group $G$ and
a countably compact $G^\star$-regular space $X$ such that $C_p(X,G)$ is not TAP (\cite[Theorem 6.8]{Sp}). The authors of \cite{Sp} have asked whether for a TAP group $G$ and a $G$-regular ($G^\star$-regular, $G^{\star\star}$-regular) compact space $X$, the group $C_p(X,G)$ must be TAP (\cite[Question 11.1]{Sp}). Our next example answers this question in the negative.

\begin{example}
Let $p$ be a prime number and $X=\{0\}\cup \{1/n:n\in \N\}$ a convergent sequence. Furthermore, let $G$ be any TAP subgroup of $\Z_p$ (for example, the cyclic group $\Z$ from Example \ref{question1} will do). Then {\em $G$ is a precompact metric TAP group and $X$ is
a compact $G^{\star\star}$-regular space such that $C_p(X,G)$ is not TAP.\/} Indeed, since
$$
\{0\}\times\prod_{n\in\N} p^nG = \{f\in C_p(X,G):
f(0)=0
\mbox{ and }
f({1}/{n})\in p^n G
\mbox { for every }
n\in\N\}\subseteq C_p(X,G)
$$
and each $p^nG$ is nontrivial,  we conclude that $C_p(X,G)$ contains a subgroup topologically isomorphic  to an infinite product
$H=\prod_{n\in\N} H_n$ of non-trivial groups $H_n=p^nG$. For every $n\in\N$, choose $h_n\in H$ such that $h_n(m)\not=0$ if and only if
$m=n$. Then the sequence $A=\{h_n:n\in\N\}$ is unconditionally $f_\omega$-productive in $H$ (and thus, in $C_p(X,G)$ as well).
Therefore, $C_p(X,G)$ is not TAP. Since $X$ is zero-dimensional (in the sense of ind), it follows that $X$ is $G^{\star\star}$-regular;
see also \cite[Proposition 2.3(iii)]{Sp}.
\end{example}

\section{Four productivity spectra of a topological group}
\label{four:spectra:section}

In this section we consider a general question of existence of (unconditionally) $f$-productive sequences and unconditionally $f$-Cauchy productive sequences for
various functions $f\in \Wi$.

\begin{proposition}
\label{infinitely:many}
Let $g\in \Wi$ be such that the set $\{n\in\N: g(n)=\omega\}$ is infinite. If a topological group $G$ contains a $g$-productive ($g$-Cauchy productive)
sequence, then $G$ contains also an $f_\omega$-productive (an $f_\omega$-Cauchy productive) sequence.
\end{proposition}

\begin{proof}
Let $\{a_k:k\in\N\}$ be a $g$-productive sequence (a $g$-Cauchy productive sequence) in $G$. Choose a strictly increasing sequence $\{k_n:n\in\N,n\ge m\}\subseteq\{n\in\N: g(n)=\omega\}$. Since $\{a_k:k\in\N\}$ is $g$-productive ($g$-Cauchy productive), the subsequence $\{a_{k_n}:n\in\N, n\ge m\}$ of the sequence $\{a_k:k\in\N\}$ is $f_\omega$-productive ($f_\omega$-Cauchy productive, respectively).
\end{proof}

\begin{proposition}
\label{finitely:many:omegas}
If a topological group $G$ contains an $f$-productive sequence (an $f$-Cauchy productive sequence) for some unbounded function
$f\in \W$, then $G$ contains a $g$-productive sequence (a $g$-Cauchy productive sequence) for every function $g\in\Wi$ such that the set $\{n\in\N: g(n)=\omega\}$ is finite.
\end{proposition}

\begin{proof}
By our assumption on $g$, there exists $m\in\N$ such that $g(n)\not=\omega$ (and thus, $g(n)\in\N$) for all $n\in\N$ with $n\ge m$.
Let $\{a_k:k\in\N\}$ be an $f$-productive sequence (an $f$-Cauchy productive sequence) in $G$ for an unbounded function $f\in\W$.
Since $f$ is unbounded, by induction on $n\in\N$ we can choose a strictly increasing sequence $\{k_n:n\in\N,n\ge m\}\subseteq\N$ such that
$f(k_n)\ge g(n)$ for every $n\in\N$ with $n\ge m$. Define $h:\N\to\N$ by $h(n)=f(k_n)$ for all $n\in\N$. Since $\{a_k:k\in\N\}$ is $f$-productive
($f$-Cauchy productive), the subsequence $\{a_{k_n}:n\in\N, n\ge m\}$ of the sequence $\{a_k:k\in\N\}$ is $h$-productive ($h$-Cauchy productive). Since $g\le^*h$, from Lemma
\ref{if:f<g}
we conclude that the sequence $\{a_{k_n}:n\in\N\}$ in $G$ is $g$-productive ($g$-Cauchy productive) as well.
\end{proof}

\begin{definition}
\label{spectrum}
Let $G$ be a topological group.
\begin{itemize}
\item[(i)]
We call the sets
$$
\mathscr{P}(G)=\{f\in\Wi: G
\mbox{ contains an
$f$-productive sequence}\}
$$
and
$$
\mathscr{P}_C(G)=\{f\in\Wi: G
\mbox{ contains an
$f$-Cauchy productive sequence}\}
$$
the {\em productivity spectrum of $G$\/} and the {\em Cauchy productivity spectrum of $G$\/},
respectively.
\item[(ii)]
We call the sets
$$
\mathscr{P}_u(G)=\{f\in\Wi: G
\mbox{ contains an unconditionally $f$-productive sequence}\}
$$
and
$$
\mathscr{P}_{uC}(G)=\{f\in\Wi: G
\mbox{ contains an unconditionally $f$-Cauchy productive sequence}\}
$$
the {\em unconditional productivity spectrum of $G$\/} and the {\em unconditional Cauchy productivity spectrum of $G$\/}, respectively.
\item[(iii)]
When $G$ is abelian, we replace the word ``productive'' with the word ``summable'' in items (i) and (ii) to get the abelian version of these notions.
\end{itemize}
\end{definition}

\begin{lemma}
\label{general:properties:of:spectra}
Let $G$ be a topological group.
\begin{itemize}
\item[(i)] $\mathscr{P}_u(G)\subseteq \mathscr{P}(G)\subseteq \mathscr{P}_C(G)$.
\item[(ii)] $\mathscr{P}_{u}(G)\subseteq \mathscr{P}_{uC}(G)\subseteq \mathscr{P}_C(G)$.
\item[(iii)] If $G$ is abelian, then $\mathscr{P}_u(G) = \mathscr{P}(G)$ and
$\mathscr{P}_{uC}(G)=\mathscr{P}_C(G)$.
\item[(iv)] If $G$ is sequentially Weil complete, then $\mathscr{P}(G)= \mathscr{P}_C(G)$
and $\mathscr{P}_{u}(G)=\mathscr{P}_{uC}(G)$.
\end{itemize}
\end{lemma}
\begin{proof}
Item (i) follows from Definitions \ref{f:productive:sets}(i), \ref{spectrum} and Lemma \ref{easy:connections}(i);
item (ii) follows from Lemma \ref{easy:connections:Cauchy}(i) and Definitions \ref{f:productive:sets}, \ref{spectrum};
item (iii) follows from Definition \ref{spectrum} and Corollary \ref{abelian:corollary};
item (iv) follows from Definition \ref{spectrum} and Lemmas \ref{easy:connections}(ii), \ref{easy:connections:Cauchy}(ii).
\end{proof}

\begin{definition}
\label{star:of:the:family}
For $\mathscr{A}\subseteq \Wi$, define $\mathscr{A}^*=\{g\in \Wi: g=^* f$ for some $f\in\mathscr{A}\}$.
\end{definition}

One can easily check that $\mathscr{A}\mapsto \mathscr{A}^*$ is a closure operator on
$\Wi$. Our next lemma says that the four spectra from Definition \ref{spectrum}
are closed sets with respect to this closure operator.

\begin{lemma}
\label{closed:under:taling:stars}
$\mathscr{P}(G)=\mathscr{P}(G)^*$, $\mathscr{P}_C(G)=\mathscr{P}_C(G)^*$,
$\mathscr{P}_u(G)=\mathscr{P}_u(G)^*$ and $\mathscr{P}_{uC}(G)=\mathscr{P}_{uC}(G)^*$.
\end{lemma}
\begin{proof}
The first two equalities follow from Definition \ref{spectrum} and Lemma \ref{if:f<g}, and the last two equalities follow from Definition \ref{spectrum} and Lemma \ref{unconditional:if:f<g}.
\end{proof}

Let
$$
\mathscr{B}=\{f\in \W: \exists\ k\in\N\
\forall\ n\in\N\
f(n)\le k\}
$$
be the set of all bounded functions.

\begin{theorem}
\label{three:alternatives}
Let $G$ be a  topological group.
\begin{itemize}
\item[(i)] Either $\mathscr{P}(G)\subseteq \mathscr{B}^*$ or $\mathscr{P}(G)\in\{\W^*, \Wi\}$.
\item[(ii)] Either $\mathscr{P}_C(G)\subseteq \mathscr{B}^*$ or $\mathscr{P}_C(G)\in\{\W^*, \Wi\}$.
\end{itemize}
\end{theorem}
\begin{proof}
(i) We consider two cases.

\medskip
{\em Case 1\/}. {\sl The set $\{n\in\N:h(n)=\omega\}$ is finite for every $h\in \mathscr{P}(G)$.\/}
 From this and Definition \ref{star:of:the:family} one easily concludes that
\begin{equation}
\label{eq:20}
\mathscr{P}(G)\subseteq\W^*.
\end{equation}
Suppose that $\mathscr{P}(G)\setminus \mathscr{B}^*\not=\emptyset$ and choose $g\in \mathscr{P}(G)\setminus \mathscr{B}^*$.
Since $g\in \mathscr{P}(G)\subseteq\W^*$, there exists $f\in\W$ such that $g=^* f$.
Since $f=^*g$ and $g\not\in\mathscr{B}^*$, the function $f\in\W$ is unbounded.
Since $g\in \mathscr{P}(G)$, there exists a $g$-productive sequence in $G$.
Since $f=^* g$, the same sequence is also $f$-productive by Lemma \ref{if:f<g}. Applying Proposition \ref{finitely:many:omegas}, we conclude that
$\W^*\subseteq \mathscr{P}(G)$. Combining this with \eqref{eq:20}, we get $\mathscr{P}(G)=\W^*$.

\medskip
{\em Case 2\/}. {\sl The set $\{n\in\N:g(n)=\omega\}$ is infinite for some $g\in \mathscr{P}(G)$.\/}
In this case $f_\omega\in \mathscr{P}(G)$ by Proposition \ref{infinitely:many}. Since $f\le f_\omega$ for every $f\in\Wi$, from Lemma \ref{if:f<g} we obtain
$\Wi\subseteq \mathscr{P}(G)$. Since the inclusion $\mathscr{P}(G)\subseteq \Wi$ holds trivially, we get
$\mathscr{P}(G)=\Wi$.

(ii) The proof similar to that of (i) is left to the reader.
\end{proof}

\begin{corollary}
\label{four:abelian:spectra}
Let $G$ be an abelian topological group. Then:
\begin{itemize}
\item[(i)] $\mathscr{P}(G)\in\{\emptyset, \mathscr{B}^*, \W^*, \Wi\}$;
\item[(ii)] $\mathscr{P}_C(G)\in\{\emptyset, \mathscr{B}^*, \W^*, \Wi\}$.
\end{itemize}
\end{corollary}
\begin{proof}
(i)
Suppose that $\mathscr{P}(G)\notin\{\emptyset, \W^*, \Wi\}$.
Then
\begin{equation}
\label{eq:16}
\mathscr{P}(G)\subseteq \mathscr{B}^*
\end{equation}
by Theorem \ref{three:alternatives}(i), and we can also fix $g\in \mathscr{P}(G)$. Then $G$ contains a $g$-productive sequence. Since $g\in \mathscr{B}^*$,
$g=^*f$ for some $f\in\mathscr{B}$. By Lemma \ref{if:f<g}, the same sequence is also $f$-productive. Applying Proposition \ref{bounded:f}, we conclude that
$\mathscr{B}\subseteq \mathscr{P}(G)$. Therefore, $\mathscr{B}^*\subseteq \mathscr{P}(G)^*=\mathscr{P}(G)$ by Lemma \ref{closed:under:taling:stars}.
 From this and \eqref{eq:16}, we obtain $\mathscr{P}(G)= \mathscr{B}^*$.

(ii) has a similar proof that is omitted.
\end{proof}

\begin{proposition}
\label{general:estimate:for:Pu}
Let $G$ be a topological group. Then:
\begin{itemize}
\item[(i)] either $\mathscr{P}_u(G)\subseteq \W^*$ or $\mathscr{P}_u(G)= \Wi$;
\item[(ii)] either $\mathscr{P}_{uC}(G)\subseteq \W^*$ or $\mathscr{P}_{uC}(G)= \Wi$.
\end{itemize}
\end{proposition}
\begin{proof}
(i) We consider two cases.

\medskip
{\em Case 1\/}. {\sl The set $\{n\in\N:f(n)=\omega\}$ is finite for every $f\in \mathscr{P}_u(G)$.\/}
 From this and Definition \ref{star:of:the:family} one easily concludes that
$\mathscr{P}_u(G)\subseteq\W^*$.

\medskip
{\em Case 2\/}. {\sl The set $S=\{n\in\N:f(n)=\omega\}$ is infinite for some $f\in \mathscr{P}_u(G)$.\/} Fix an injection $\sigma:\N\to
S$. Since $f\in \mathscr{P}_u(G)$, there exists an unconditionally $f$-productive sequence $A=\{a_n:n\in\N\}$ in $G$. Let
$\varphi:\N\to\N$ be a bijection. Since $A$ is unconditionally $f$-productive, it follows from Lemma \ref{AP:set:has:AP:subsets}
that the sequence $B_\varphi=\{a_{\sigma\circ\varphi(n)}:n\in\N\}$ is $(f\circ\sigma\circ\varphi)$-productive. Since $\sigma\circ\varphi(\N)\subseteq S$, from our assumption we conclude that $f\circ\sigma\circ\varphi=f_\omega$. Therefore, $B_\varphi$ is $f_\omega$-productive. Since bijection $\varphi$
was chosen arbitrarily, it follows that the sequence
$\{a_{\sigma(n)}:n\in\N\}$ is unconditionally $f_\omega$-productive.
Since $f\le f_\omega$ for every $f\in \Wi$, from Lemma
\ref{unconditional:if:f<g} we conclude that $G$ has an
unconditionally $f$-productive sequence for every $f\in\Wi$. Thus,
$\mathscr{P}_u(G)=\Wi$.

A similar proof of (ii) is omitted.
\end{proof}

\begin{corollary}
\label{TAP:in:terms:of:spectra}
A topological group $G$ is TAP if and only if $\mathscr{P}_u(G)\subseteq \W^*$.
\end{corollary}
\begin{proof}
By Definitions \ref{def:TAP} and \ref{spectrum}, $G$ is TAP if and only if
$f_\omega\not\in \mathscr{P}_u(G)$, and since $f\le f_\omega$ for every $f\in\Wi$, the latter condition is equivalent to
$\mathscr{P}_u(G)\not=\Wi$ by Lemma \ref{unconditional:if:f<g}.
Finally, $\mathscr{P}_u(G)\not=\Wi$ is equivalent to $\mathscr{P}_u(G)\subseteq \W^*$ by Proposition \ref{general:estimate:for:Pu}(i).
\end{proof}

Our next example shows that all four cases in Corollary \ref{four:abelian:spectra}(i)
can be realized even for some {\em separable metric\/} abelian group.

\begin{example}
\begin{itemize}
\item[(i)]
Every metric abelian group $G$ of size $<\cont$ does not contain
$f_1$-summable sequences by Corollary \ref{|G|<cont:is:NAP}, and so $\mathscr{P}(G)=\emptyset$ by Corollary \ref{four:abelian:spectra}(i).
\item[(ii)]
Let $H$ be the group constructed in Example \ref{bounded:functions} below. Then $\mathscr{P}(H)\not=\emptyset$ by Example \ref{bounded:functions}(i) and $\mathscr{P}(H)\not\in\{\W^*, \Wi\}$ by Example \ref{bounded:functions}(ii). Therefore,  $\mathscr{P}(H)=\mathscr{B}^*$ by Corollary \ref{four:abelian:spectra}(i).
\item[(iii)]
 The group $\R$ of real numbers contains $g$-summable sequences for all functions $g\in\W$ by Corollary \ref{no:f-product:seq:in:nondiscrete:metr:group},
so $\W^*\subseteq \mathscr{P}(\R)^*=\mathscr{P}(\R)$ by Lemma \ref{closed:under:taling:stars}.
Since $\R$  does not contain any $f_\omega$-summable sequence by Corollary \ref{weil:NSS:iff:TAP}, we have $\mathscr{P}(\R)\not=\Wi$.
Thus, $\mathscr{P}(\R)=\W^*$ by Corollary \ref{four:abelian:spectra}(i).
\item[(iv)]
For each $n\in\N$ define $a_n\in \Z^\N$ by $a_n(n)=1$ and $a_n(m)=0$ for $m\in\N\setminus\{n\}$. One can easily see that the sequence $\{a_n:n\in\N\}$ is
$f_\omega$-summable. By Lemma \ref{if:f<g}, $\Z^\N$ contains an $f$-summable sequence for every $f\in\Wi$. Therefore, $\mathscr{P}(\Z^\N)=\Wi$.
\end{itemize}
\end{example}

We now summarize our principal results from previous sections in the language of productivity spectra.

\begin{theorem}
Let $G$ be a metric group.
\begin{itemize}
\item[(i)] $G$ is discrete if and only if $\mathscr{P}_C(G)=\emptyset$;
\item[(ii)] If $G$ is non-discrete, then either $\mathscr{P}_C(G)=\W^*$ or $\mathscr{P}_C(G)=\Wi$.
\item[(iii)] $\mathscr{P}_C(G)=\W^*$ if and only if $G$ is NSS and non-discrete.
\end{itemize}
In particular, a non-discrete sequentially Weil complete metric group $G$ is NSS if and only if $\mathscr{P}(G)=\W^*$.
\end{theorem}
\begin{proof}
(i) A discrete group does not have non-trivial convergent sequences,
so $\mathscr{P}_C(G)=\emptyset$ by Lemma \ref{AP:is:null:sequence}
and Definition \ref{spectrum}.
If $G$ is non-discrete, then $\mathscr{P}_C(G)\not=\emptyset$ by Theorem \ref{non1-TAP}.

(ii) Since $\W^*\subseteq \mathscr{P}_C(G)$ by Theorem \ref{non1-TAP},
the conclusion follows from Theorem \ref{three:alternatives}(ii).

(iii) If $\mathscr{P}_C(G)=\W^*$, then $G$ is non-discrete by (i), and since $f_\omega\not\in \W^*=\mathscr{P}_C(G)$, $G$ does not contain an $f_\omega$-Cauchy productive sequence; thus, $G$ is NSS by Theorem \ref{metric:NACP:iff:NSS}. Suppose now that $G$ is non-discrete and NSS. Then $\mathscr{P}_C(G)\not=\emptyset$ by (i), and $G$ does not have an $f_\omega$-Cauchy productive sequence by Theorem \ref{metric:NACP:iff:NSS}. In particular, $\mathscr{P}_C(G)\not=\Wi$, and so so $\mathscr{P}_C(G)=\W^*$ by (ii).

Finally, the ``in particular'' part of our theorem follows from (iii) and Lemma \ref{general:properties:of:spectra}(iv).
\end{proof}

\begin{theorem}
Let $G$ be either a non-discrete locally compact or an infinite $\omega$-bounded group. Then
either $\mathscr{P}_u(G)=\mathscr{P}_{uC}(G)=\W^*$ or $\mathscr{P}_u(G)=\mathscr{P}_{uC}(G)=\Wi$.
Moreover, $\mathscr{P}_u(G)=\W^*$ if and only if $G$ is NSS (if and only if $G$ is a Lie group).
\end{theorem}

\begin{proof}
Note that $G$ is sequentially Weil complete, so $\mathscr{P}_u(G)=\mathscr{P}_{uC}(G)$ by Lemma \ref{general:properties:of:spectra}(iv).
The inclusion $\mathscr{W}^*\subseteq \mathscr{P}_u(G)$ follows from Theorem \ref{zerodim:TAP:compact:is:finite} and Corollary \ref{omega:bounded}.
Therefore,  $\mathscr{P}_u(G)\in\{\W^*,\Wi\}$ by Proposition \ref{general:estimate:for:Pu}(i). By Corollary \ref{TAP:in:terms:of:spectra},
$\mathscr{P}_u(G)=\W^*$ if and only if $G$ is TAP. Now the final claim of our theorem follows from
\cite[Theorems 10.8 and 10.9]{DSS}.
\end{proof}

For a group $H$ constructed in Theorem \ref{TAP:group:with:ap:sequence} we have
$\mathscr{P}_{uC}(H)=\mathscr{P}_C(H)=\mathscr{P}(H)=\Wi$ but $\mathscr{P}_u(H)\subseteq \W^*$. We do not know if $\mathscr{P}_u(H)\not=\emptyset$.

\begin{question}
What are possible values of $\mathscr{P}(G)$ for an arbitrary (not necessarily abelian) topological group $G$?
\end{question}

\section{Eliminating $f$-productive sequences in products}
\label{Eliminating f-productive sequences in products}

\begin{definition}
\label{definition:of:a:seminorm}
Let $\R^+$ denote the set of all non-negative real numbers. Recall that a {\em seminorm\/} on a group $G$ is a function $\nu:G\to\R^+$
satisfying the following conditions:
\begin{itemize}
\item[(S1)] $\nu(ab)\le\nu(a)+\nu(b)$ for every $a,b\in G$;
\item[(S2)] $\nu(a)=\nu(a^{-1})$ for each $a\in G$.
\end{itemize}
\end{definition}

We will need the following folklore lemma.

\begin{lemma}\label{second:seminorm:inequality}
If $\nu:G\to\R^+$ is a seminorm on a group $G$ and $a,b\in G$, then $|\nu(a)-\nu(b)|\le\nu(ab)$.
\end{lemma}

\begin{lemma}
\label{technical:lemma}
Suppose that $I\not=\emptyset$ is a set, $g\in \Wi$, $\nu$ is a seminorm on a discrete group $D$, $H$ is a subgroup of the Tychonoff product $D^I$ such that the set $\{\nu(h(i)):i\in I\}$ is bounded for every $h\in H$. Let $A=\{a_m:m\in\N\}$ be a faithfully indexed sequence of elements of $H$, $M$ be an infinite subset of $\N$, $\{z_m:m\in\N\}\subseteq\Z$ and $\{i_m:m\in\N\}\subseteq I$. Assume also that for every $m\in\N$ the following three conditions hold:
\begin{enumerate}
\item[(i$_m$)] $|z_m|\le g(m)$;
\item[(ii$_m$)] $a_m(i_k)=e$ for every $k\in M$ such that $0\le k<m$;
\item[(iii$_m$)]
$\nu\left(\prod_{j=0}^ma_j^{z_j}(i_m)\right)\ge m$ whenever $m\in M$.
\end{enumerate}
Then $A$ is not a $g$-productive sequence in $H$.
\end{lemma}

\begin{proof}
Suppose that $A$ is a $g$-productive sequence in $H$. Since (i$_m$) holds for every $m\in\N$, the sequence
$\left\{\prod_{j=0}^k a_j^{z_j}:k\in\N\right\}$ must have a limit $a\in H$. Since (ii$_m$) holds for every $m\in\N$,
we must have $a(i_m)=\prod_{j=0}^m a_j^{z_j}(i_m)$ for every $m\in M$. Combining this with (iii$_m$), we conclude that $\nu(a(i_m))\ge m$ for each
$m\in M$. Since $M$ is infinite, this means that the  set $\{\nu(a(i)):i\in I\}$ is unbounded. Since $a\in H$, this contradicts the assumption of our lemma.
\end{proof}

\begin{lemma}\label{impossible:case:for:unbounded:seminorm}
Let $I\not=\emptyset$ be a set, $f\in \Wi$ and $\nu$ be a seminorm on a discrete group $D$. Suppose that  a subgroup $H$ of the Tychonoff
product $D^I$ and  a faithfully indexed sequence $\{h_n:n\in\N\}$ of elements of $H$ satisfy the following conditions:
\begin{itemize}
\item[(i)]
the set $\{\nu(h(i)):i\in I\}$ is bounded for every $h\in H$;
\item[(ii)]
the set
\begin{equation}
\label{eq:N_r} N_r=\{n\in \N: \exists i\in I\ \exists z\in\Z\
\left(
|z|\le f(n)
\mbox{ and }
\nu(h_n(i)^z)\ge r\right)\}
\end{equation}
is infinite for every $r\in\R$.
\end{itemize}
Then $\{h_n:n\in\N\}$ is not an $f$-productive sequence in $H$.
\end{lemma}

\begin{proof}
Assume that $\{h_n:n\in\N\}$ is an $f$-productive sequence in $H$. Then $\{h_n:n\in\N\}$ converges to the identity of $H$ by Lemma
\ref{AP:is:null:sequence}. Since $H\subseteq D^I$ and $D$ is discrete, it follows that
\begin{equation}
\label{twin:of:eq:8} \{n\in \N: h_n(i)\not=e\} \mbox{ is finite for
every } i\in I.
\end{equation}

By induction on $m\in\N$ we will choose $n_m\in\N$, $z_m\in\Z$ and $i_m\in I$ satisfying the following conditions:
\begin{enumerate}
\item[(1$_m$)] $n_m>n_{m-1}$ for $m\ge 1$;
\item[(2$_m$)] $|z_m|\le f(n_m)$;
\item[(3$_m$)] $h_{n_m}(i_k)=e$ for every $k=0,1,\dots,m-1$;
\item[(4$_m$)]
$\nu(g_m(i_m))\ge m$, where $g_m=\prod_{j=0}^mh_{n_j}^{z_j}$.
\end{enumerate}
Choose $n_0\in\N$, $i_0\in I$ arbitrarily, put $z_0=0$ and note that conditions (1$_0$), (2$_0$), (3$_0$) and (4$_0$) are satisfied.

Let  $m\in\N\setminus\{0\}$, and assume that $n_j\in\N$, $z_j\in\Z$ and $i_j\in I$ satisfying conditions (1$_j$), (2$_j$), (3$_j$) and (4$_j$)
have already been chosen for every $j=0,1,\dots, m-1$. We will choose $n_m\in\N$, $z_m\in\Z$ and $i_m\in I$ satisfying conditions
(1$_m$), (2$_m$), (3$_m$) and (4$_m$). First, apply (i) to choose $s\in \N$ such that
\begin{equation}
\label{psi:is:bounded:on:h_{m-1}} \nu(g_{m-1}(i))\le s \mbox{ for every } i\in I.
\end{equation}
 From \eqref{twin:of:eq:8} it follows that the set
\begin{equation}
\label{eq:F_{m-1}} F_{m-1}=\bigcup_{k=0}^{m-1}\{n\in\N:h_n(i_k)\neq
e\}
\end{equation}
 is finite, so we can apply (ii) to choose $n_m\in N_{s+m}$ such that $n_m>\max F_{m-1}$
and $n_m>n_{m-1}$. In particular, (1$_m$) is satisfied. Since
$n_m\not\in F_{m-1}$, from \eqref{eq:F_{m-1}} we get (3$_m$).

Since $n_m\in N_{s+m}$, from \eqref{eq:N_r} it follows that there exist $i_m\in I$ and $z_m\in\Z$ such that (2$_m$) holds and
\begin{equation}
\label{eq:psi(a_m(i_m))>n_m:new} \nu(h_{n_m}(i_m)^{z_m})\ge s+m.
\end{equation}
 From \eqref{psi:is:bounded:on:h_{m-1}} and
\eqref{eq:psi(a_m(i_m))>n_m:new} we get $
\nu\left(h_{n_m}(i_m)^{z_m}\right)-\nu(g_{m-1}(i_m))\ge s+m-s=m. $ Since $\nu$ is a seminorm, combining this with Lemma
\ref{second:seminorm:inequality}, we obtain
$$
\nu(g_m(i_m))=\nu(g_{m-1}(i_m) h_{n_m}(i_m)^{z_m})\ge \left|\nu\left(h_{n_m}(i_m)^{z_m}\right)-\nu(g_{m-1}(i_m))\right|\ge m.
$$
Thus, (4$_m$) holds. This finishes our inductive construction.

For $m\in\N$ let $a_m=h_{n_m}$. Define the function $g\in\Wi$ by
$g(m)=f(n_m)$ for $m\in\N$. We claim that $I$, $g$, $\nu$, $H$,
$A=\{a_m:m\in\N\}$, $M=\N$, $\{z_m:m\in\N\}$ and $\{i_m:m\in\N\}$
satisfy the assumptions of Lemma \ref{technical:lemma}. Indeed,
(2$_m$) implies the condition (i$_m$) of Lemma
\ref{technical:lemma}, (3$_m$) implies (ii$_m$) and (4$_m$) implies
(iii$_m$). From Lemma \ref{technical:lemma} we conclude that $A$ is
not a $g$-productive sequence in $H$. Since (1$_m$) holds for every
$m\in\N$, $A=\{h_{n_m}:m\in\N\}$ is a subsequence of the sequence
$\{h_n:n\in\N\}$. Since $g(m)=f(n_m)$ for every $m\in\N$, it follows
that $\{h_n:n\in\N\}$ is not an $f$-productive sequence in $H$, a
contradiction.
\end{proof}

We
finish this section with
an example demonstrating an application of Lemma
\ref{impossible:case:for:unbounded:seminorm}.

\begin{example}
\label{bounded:functions} Let
$H=\{h\in\Z^\N: \exists\ k\in\N\
\forall\ i\in\N\ |h(i)|\le k\}$
be the subgroup of the linear metric
group $\Z^\N$ (equipped with the Tychonoff product topology).
For each $n\in\N$ define $a_n\in H$ by
$a_n(i)=1$ if $i=n$ and $a_n(i)=0$ otherwise ($i\in\N$).
Then:
\begin{itemize}
\item[(a)]
{\em
the sequence $A=\{a_n:n\in\N\}\subseteq H$ is
$f$-summable in $H$
for every bounded
function
$f\in\W$;\/}
\item[(b)]
{\em
$H$ does not contain  $f$-summable
sequences for any unbounded function
$f\in\W$\/}.
\end{itemize}
Indeed,
(a) is straightforward.
To check (b),
assume that the function
$f\in \W$
is unbounded and $\{h_n:n\in\N\}\subseteq H$ is an arbitrary faithfully
indexed sequence in $H$.
Note that the function $\nu:\Z\to\N$ defined by $\nu(z)=|z|$
is a seminorm on $\Z$.
We claim
that conditions (i) and (ii) of Lemma
\ref{impossible:case:for:unbounded:seminorm} are satisfied
with $D=\Z$ and $I=\N$.
Indeed, (i) is satisfied by our definitions of $H$ and $\nu$. To show that
condition (ii) holds, fix $r\in\R$. Since $f\in\Wi$ is unbounded, the set $A_r=\{n\in\N:f(n)>r\}$ is infinite. Furthermore,
if $n\in A_r$ and $h_n\neq e$, then there exists $i\in\N$ such that $|h_n(i)|\ge 1$ and consequently, $|f(n)h_n(i)|\ge f(n)> r$.
Thus, condition (ii) of Lemma \ref{impossible:case:for:unbounded:seminorm} is satisfied
as well.
Applying Lemma \ref{impossible:case:for:unbounded:seminorm}, we conclude that
the sequence $\{h_n:n\in\N\}$ is not $f$-summable in $H$.
\end{example}

\begin{remark}
Example \ref{bounded:functions} shows that:
\begin{itemize}
\item[(i)] ``$f$-Cauchy productive'' cannot be replaced by ``$f$-productive'' in Theorems \ref{non1-TAP} and \ref{AP-Cauchy:seq:in:lin:groups},
\item[(ii)] Weil completeness cannot be omitted in Corollaries \ref{no:f-product:seq:in:nondiscrete:metr:group},
\ref{Weil:complete:linear:group} and \ref{linear:complete:corollary}, and
\item[(iii)] sequential Weil completeness is essential in Corollaries \ref{NACP:iff:no:conv:sequences}
and \ref{sequentially:complete:liner:TAP:group:is:discrete}.
\end{itemize}
\end{remark}

\section{A necessary background on free groups}
\label{background:on:free:groups}

\begin{definition}
Let
$k\in\N$ and  let $s=(x_0, x_1,\dots, x_{k})$ be a finite sequence  of elements of a set $X$.
\begin{itemize}
\item[(i)] The sequence $s$  is said to be a {\em \sequence\ } if $x_i\neq x_{i+1}$ for every integer $i$ with $0\le i< k$.
We denote by $\Seq(X)$ the family of all sequences with distinct neighbors in $X$.
\item[(ii)]
For an integer $k'\ge 0$ and natural numbers $i_0,i_1,\dots,i_{k'}$ such that $0\leq i_0< i_1<  \ldots  < i_{k'}\leq  k$,
we call $s'=(x_{i_0}, x_{i_1},\ldots , x_{i_{k'}})$ a {\em subsequence} of $s$.
\end{itemize}
\end{definition}

Thereafter, $F(X)$ denotes the free group with an alphabet $X$. Recall that every element $w\in F(X)\setminus\{e\}$ has a unique representation
\begin{equation}
\label{word:a}
w=x_1^{z_1}x_2^{z_2}\dots x_{k}^{z_{k}},
\end{equation}
 where  $k\in\N\setminus\{0\}$,  $\support{w}=(x_1,x_2,\dots,x_k)\in\Seq(X)$
and
$z_1,z_2,\dots,z_k\in \Z\setminus\{0\}$. We call \eqref{word:a}
the {\em canonical representation\/} of $w$ and the integer $k$ the {\em length\/} of $w$.
The length of $w$ will be denoted by $l_w$.  For every integer $i$ with $1\le i\le l_w$ we define
$$
\letter{w}{i}=x_i,
\
\power{w}{i}=z_i
\
\mbox{ and }
w[i]=x_i^{z_i}.
$$
In particular,
$$
\support{w}=(\letter{w}{1},\letter{w}{2},\dots,\letter{w}{l_w}).
$$
We define $l_e=0$ and identify $\support{e}$ with the empty sequence.

We call elements of $F(X)$ {\em words\/}. A word of length $1$ is called a {\em monom\/}.
An element of the set $\{w[i]:i\in \N, 1\le i\le l_w\}$ is called a
{\em monom of $w$\/}.
Monoms $x_1^{z_1}$ and $x_2^{z_2}$ are called {\em independent\/} provided that
$x_1\neq x_2$; otherwise we call them {\em dependent\/}. In the former case the product of
$x_1^{z_1}$ and $x_2^{z_2}$ in $F(X)$ coincides with the word
$x_1^{z_1}x_2^{z_2}$, and in the latter case the product of
$x_1^{z_1}$ and $x_2^{z_2}$ in $F(X)$ is equal to the word
$x_1^{z_1+z_2}$. It is a monom if $z_1+z_2\neq 0$ and it is equal to $e$ otherwise.
In particular,
\eqref{word:a} becomes
$w=w[1]\cdot\ldots\cdot w[l_w]$; that is, $w$ is a product, taken in  $F(X)$, of the independent monoms $w[i]$ of $w$.

Given a word $w\in F(X)$ as in
\eqref{word:a},
a {\em subword\/} of $w$ is a word of the form
$$
u=w[i]\cdot w[i+1]\cdot\ldots \cdot w[j]
$$
for some integers $i,j$ with $1\le i\le j\le l_w$.
If $i=1$, we call $u$ an {\em initial subword\/} of $w$, and if $i=j=1$, we call $u$ an {\em initial monom\/} of $w$.
Similarly, if $j=l_w$, we call $u$ a {\em final subword\/} of $w$, and if $i=j=l_w$, we call $u$ a {\em final monom\/} of $w$.
We consider $e$ to be both an initial and a final subword of every word $w$.

We say that the product $w_0w_1\dots w_n$ of words
$w_0,w_1,\dots,w_n\in F(X)$ is {\em cancellation-free\/}  provided
that the final monom of $w_i$ and the initial monom of $w_{i+1}$ are
independent for every integer $i$ such that $0\le i< n$.

\begin{definition}
\label{def:of:c}
For $v,w\in F(X)$ we denote by $\cancel{v}{w}$ the initial subword $u$
of $w$ of maximal length such that $u^{-1}$ is a final subword of $v$.
\end{definition}

Let $v'$ be the initial subword of $v$ such that $v=v'\cdot
(\cancel{v}{w})^{-1}$. Similarly, let $w'$ be the final subword of
$w$ such that $w=(\cancel{v}{w})\cdot w'$. Then
\begin{equation}
\label{eq:v:w}
vw=\left(v'\cdot (\cancel{v}{w})^{-1}\right)\left((\cancel{v}{w})\cdot w'\right)
=
v'\cdot w'.
\end{equation}
We say that all monoms in $(\cancel{v}{w})^{-1}$ and $\cancel{v}{w}$
{\em cancel completely\/} in the product $vw$. If the product $v'\cdot w'$ on the right of \eqref{eq:v:w} is not cancellation-free,
then from the maximality of length of $u$ in Definition \ref{def:of:c} it follows that the final monom $v'[l_{v'}]$ of $v'$
and the initial monom $w'[1]$ of $w'$ are dependent and $v'[l_{v'}]\cdot w'[1]\not=e$. That is, $\letter{v'}{l_{v'}}=\letter{w'}{1}=x_{v,w}$ for some $x_{v,w}\in X$
and  $z_{v,w}=\power{v'}{l_{v'}}+\power{w'}{1}\not=0$. Furthermore,
\begin{equation}
\label{eq:vw:2}
vw=v[1]\cdot\ldots \cdot v[l_{v'}-1]\cdot x_{v,w}^{z_{v,w}}\cdot w'[2]\cdot\ldots\cdot w'[l_{w'}],
\end{equation}
and the product on the right in \eqref{eq:vw:2} is cancellation-free. (Here we let $v[1]\cdot\ldots \cdot v[l_{v'}-1]=e$ when $l_{v'}=1$
and $w'[2]\cdot\ldots\cdot w'[l_{w'}]=e$ when $l_{w'}=1$.) In this case, we say that the final monom of $v'$ and the initial monom of
$w'$ {\em cancel partially\/} in the product $vw$. Finally, a monom of $v$ and $w$ that neither cancels completely nor cancels partially
in the product $vw$ is said to {\em remain unchanged\/} in this product.

\begin{definition}
Let $X$ be a set. For every $x\in X$, define the function $\mu_x:F(X)\to\N\subseteq\R^+$ by
\begin{equation}
\label{def:eq:phi}
\mu_x(w)=\max\left(\{|\power{w}{i}|: i\in \N,\
1\le i\le l_w,\ \letter{w}{i}=x\}\cup\{0\}\right)
\end{equation}
for
$w\in F(X)$.
\end{definition}

An alternative reformulation of the meaning of the formula \eqref{def:eq:phi} can be given as follows:
$\mu_x(w)$ is the maximal absolute value $|z|$ of the power $z$ of a monom $x^z$ of $w$, if such a monom exists, or $0$ otherwise.

\begin{lemma}
\label{a(g):equation:1}
Let $X$ be a set. For every $x\in X$, the function $\mu_x$
is a
seminorm on $F(X)$.
\end{lemma}
\begin{proof}
Fix $x\in X$. Obviously, $\mu_x(w)\ge 0$
for every $w\in F(X)$. The condition (S2) from Definition
\ref{definition:of:a:seminorm} is obvious, so it remains only to
check condition (S1). Fix $a,b\in F(X)$. Let $x^z$ be an arbitrary monom of $ab$.
(If no such a monom exists, then (S1) is trivially satisfied.)
If this monom comes from a monom that remains unchanged in the
product of $a$ and $b$, then
$$
|z|\le \max\left\{\mu_x(a),\mu_x(b)\right\}\le \mu_x(a)+ \mu_x(b).
$$
Otherwise, the monom $x^z$ must be the result of the partial
cancellation of some monom $x^{z_a}$ of $a$ with some monom $x^{z_b}$
of $b$. That is, $x^z=x^{z_a+z_b}$ and
$$
|z|=|z_a+z_b|\le |z_a|+|z_b|\le \mu_x(a)+ \mu_x(b).
$$
This finishes the proof of (S1).
\end{proof}

Our last lemma in this section shall be needed only in Section \ref{proof:of:item:v}, so the reader can safely skip it at first reading.

Suppose that $w_0,w_1,\dots,w_m$  are words such that each $w_j$ contains a distinguished
monom $x_j^{z_j}$ that may be called its ``mountain'', because we additionally assume that $|z_j|$ is
at least twice as big as the absolute value $|z|$ of the power of {\em every\/} monom $x_j^{z}$
of {\em any\/} $w_k$ with $k\not = j$. Then  the ``in particular'' part of  our  next lemma guarantees that
each ``mountain'' monom $x_j^{z_j}$ of $w_j$ is ``high enough'' to ensure that it does not cancel  completely in the product $w_0w_1\dots w_m$.
The somewhat technical conditions (i$_m$)--(iv$_m$) are needed only for  carrying an inductive argument.

\begin{lemma}
\label{real:final:step} Assume that $X$ is a set, $m\in\N$, $w_0,w_1,\dots,w_m\in F(X)$, $x_0,x_1,\allowbreak\dots,\allowbreak x_m\in X$ and
\begin{equation}
\label{mountain:eq}
\mu_{x_j}(w_j)>2\mu_{x_j}(w_k)
\mbox{ whenever }
j,k\in\N, j\le m,
k\le m
\mbox{ and }
j\not = k.
\end{equation}
Let $\mu_{x_0}(w_0)>0$. Define $v_m=w_0 w_1 \dots w_m$. Then there exists an integer $n_m$
such that:
\begin{itemize}
\item[(i$_m$)]  $1\le n_m\le l_{v_m}$;
\item[(ii$_m$)] $v_m[n_m]=x_m^{z_m}$, where $2|z_m|>\mu_{x_m}(w_m)$;
\item[(iii$_m$)] $v_m[n_m+1]\cdot v_m[n_m+2]\cdot\ldots\cdot v_m[l_{v_m}]$ is a final subword of $w_m$;
\item[(iv$_m$)] $(x_0,x_1,\dots,x_m)$ is a subsequence of the sequence
$(\letter{v_m}{1}, \letter{v_m}{2},\dots, \letter{v_m}{n_m})$.
\end{itemize}
In particular, $(x_0,x_1,\dots,x_m)$ is a subsequence of the sequence $\support{v_m}$.
\end{lemma}

\begin{proof}
We will prove our lemma by induction on $m\in\N$. By assumption of our lemma, $\mu_{x_0}(w_0)>0$, so there exists an integer $n_0$ such that $1\le n_0\le l_{w_0}$ and
$w_0[n_0]=x_0^{z_0}$, where $|z_0|=\mu_{x_0}(w_0)>0$. A straightforward check shows that this $n_0$ satisfies conditions
(i$_0$)--(iv$_0$).

Let us make an inductive step from $m\in\N$ to $m+1$. Fix $w_0,w_1,\dots,w_{m+1}\in F(X)$ and $x_0,x_1,\dots,x_{m+1}\in X$ satisfying
\eqref{mountain:eq}, with $m$ replaced by $m+1$. Our inductive assumption allows us to choose an integer $n_m$
satisfying conditions (i$_m$)--(iv$_m$). To finish the inductive step, we must find $n_{m+1}\in\N$  satisfying conditions (i$_{m+1}$)--(iv$_{m+1}$).

It easily follows from \eqref{mountain:eq} that
\begin{equation}
\label{xs:are:different}
x_m\not=x_{m+1}.
\end{equation}

\begin{claim}
\label{claim:1:new}
\begin{itemize}
\item[(a)]
$\letter{v_{m+1}}{n_m}=\letter{v_m}{n_m}=x_m$.
\item[(b)]
$v_{m+1}[i]=v_m[i]$ for every integer $i$ with $1\le i<n_m$;
that is, each monom to the left of $\letter{v_m}{n_m}$
remains unchanged in the product $v_mw_{m+1}$.
\end{itemize}
\end{claim}
\begin{proof}
If the monom $v_m[n_m]=x_m^{z_m}$ from (ii$_m$) remains unchanged in the product $v_mw_{m+1}$, then our claim holds trivially. Assume now
that this monom undergoes a cancellation in the product $v_mw_{m+1}$ with some monom $x_m^z$ of $w_{m+1}$. Since $2|z|\le
2\mu_{x_m}(w_{m+1})<\mu_{x_m}(w_m)<2|z_m|$ by \eqref{mountain:eq} and (ii$_m$), we get  $|z_m+z|\ge |z_m|-|z|>0$. That is, the
cancellation in question is only partial and results in the monom $x_m^{z_m+z}$ of $v_{m+1}=v_mw_{m+1}$. In particular, both (a) and (b) hold.
\end{proof}

\begin{claim}
\label{claim:2:new}
There exists an integer $n_{m+1}>n_m$ satisfying conditions (i$_{m+1}$), (ii$_{m+1}$) and (iii$_{m+1}$).
\end{claim}
\begin{proof}
 From \eqref{mountain:eq} we  get $\mu_{x_{m+1}}(w_{m+1})>2\mu_{x_m+1}(w_m)\ge 0$, so $w_{m+1}$ contains a monom $x_{m+1}^p$, where $|p|=\mu_{x_{m+1}}(w_{m+1})>0$. The proof now branches into two cases.

\medskip
{\em Case 1\/}. {\sl The monom $x_{m+1}^p$ of $w_{m+1}$ remains unchanged in the product $v_mw_{m+1}$\/}.
Then $x_{m+1}^p=v_{m+1}[n_{m+1}]$ for some integer $n_{m+1}$ satisfying (i$_{m+1}$),
so $z_{m+1}=p$ and $2|z_{m+1}|=2|p|>|p|=\mu_{x_{m+1}}(w_{m+1})$. This proves (ii$_{m+1}$).
Condition (iii$_{m+1}$) trivially holds as well.
 From Claim \ref{claim:1:new}(b) it follows that $n_m\le n_{m+1}$.
Since $v_{m+1}[n_m]=v_m[n_m]=x_m\not=x_{m+1}=v_{m+1}[n_{m+1}]$ by
Claim \ref{claim:1:new}(a), \eqref{xs:are:different} and (ii$_{m+1}$),
we conclude that $n_m\not=n_{m+1}$.
Thus, $n_m<n_{m+1}$.

\medskip
{\em Case 2\/}. {\sl The monom $x_{m+1}^p$ of $w_{m+1}$ undergoes a cancellation in the product $v_mw_{m+1}$ with some monom
$v_m[n_{m+1}]=x_{m+1}^q$ of $v_m$ for a suitable integer $n_{m+1}$.\/} Then $n_m\le n_{m+1}$ by Claim \ref{claim:1:new}(b).
Since $\letter{v_m}{n_m}=x_m\not=x_{m+1}=\letter{v_m}{n_{m+1}}$ by (ii$_m$) and \eqref{xs:are:different}, we have $n_m\neq n_{m+1}$,
which yields $n_m< n_{m+1}$. From (iii$_m$) it now follows that $v_m[n_{m+1}]=x_{m+1}^q$ is a monom of $w_m$. In particular,
\begin{equation}
\label{y:and:mu}
|q|\le \mu_{x_{m+1}}(w_m).
\end{equation}
Therefore,
\begin{equation}
\label{eq:p+y}
|p+q|\ge |p|-|q|=\mu_{x_{m+1}}(w_{m+1})-|q|>2\mu_{x_{m+1}}(w_m)-|q|\ge \mu_{x_{m+1}}(w_m)\ge 0
\end{equation}
by \eqref{mountain:eq}. This means that the cancellation in question
is only partial and results in the monom
$v_{m+1}[n_{m+1}]=x_{m+1}^{p+q}$ of $v_{m+1}=v_mw_{m+1}$. In
particular, (i$_{m+1}$) and (iii$_{m+1}$) hold.

Define $z_{m+1}=p+q$.  Since $2|q|\le 2 \mu_{x_{m+1}}(w_m)<\mu_{x_{m+1}}(w_{m+1})$ by \eqref{y:and:mu} and
\eqref{mountain:eq},
we get
$$
2|z_{m+1}|=2|p+q|\ge
2\mu_{x_{m+1}}(w_{m+1})-2|q|>\mu_{x_{m+1}}(w_{m+1}).
$$
This proves (ii$_{m+1}$).
\end{proof}

It remains only to check that the integer $n_{m+1}$  from Claim
\ref{claim:2:new} satisfies the condition (iv$_{m+1}$). By (iv$_m$)
and Claim \ref{claim:1:new}, the sequence $(x_0,x_1,\dots,x_m)$ is a
subsequence of the sequence $(\letter{v_{m+1}}{1},
\letter{v_{m+1}}{2},\allowbreak \dots,\allowbreak
\letter{v_{m+1}}{n_m})$. Since $x_{m+1}=\letter{v_{m+1}}{n_{m+1}}$
by (ii$_{m+1}$),  and $n_m<n_{m+1}$ by Claim \ref{claim:2:new}, this
yields (iv$_{m+1}$). The inductive step is complete.
\end{proof}

\section{Seminorms on free groups arising from linear orders}
\label{seminorms:from:linear:order}

\begin{definition}
Let $(X,<)$ be a linearly ordered set, $k\ge 0$ an integer and $s=(x_0, x_1, ..., x_k)$ a sequence of elements of $X$. Define
$$
\lma{s}=\{i\in\N: 0\le i\le k
\mbox{ and }
(i\not=0\to x_{i-1}<x_i)\wedge (i\not=k\to x_i>x_{i+1})\},
$$
$$
\lmi{s}=\{i\in\N: 0\le i\le k
\mbox{ and }
(i\not=0\to x_{i-1}>x_i)\wedge (i\not=k\to x_i<x_{i+1})\}
$$
and $\ext{s}=\lma{s}\cup\lmi{s}$. If $i\in \lma{s}$, then the $i$th element $x_i$ of $s$ will be called a {\em local maximum of $s$\/}.
Similarly, if $i\in \lmi{s}$, then the $i$th element $x_i$ of $s$ will be called a {\em local minimum of $s$\/}. Finally, if $i\in
\ext{s}$, then $x_i$ is said to be a {\em local extremum of $s$\/}.

For the empty sequence $\emptyset$, we define $\lmi{\emptyset}=\lma{\emptyset}=\emptyset$, so $\ext{\emptyset}=\emptyset$.
\end{definition}

\begin{example}\label{First:example}  Let $(X,<)$ be a linearly ordered set.
\begin{itemize}
\item[(a)]
If $k\in\N$ and $s=(x_0, x_1, ..., x_k)\in\Seq(X)$, then $0,k\in \ext{s}$; that is, the first and the last element of a sequence $s\in\Seq(X)$ are always
local extrema of $s$.
\item[(b)] If $k\in\N$ and $s=(x_0,x_1,\dots,x_k)$ is a sequence of elements of $X$
such that $x_0<x_1<\dots<x_k$, then $x_0$ is the only local minimum of $s$ and $x_k$ is the only local maximum of $s$; in particular, $\ext{s}=\{0,k\}$.
\item[(c)] If $x\ne y$ are two distinct elements of $X$, then for every $n\in\N\setminus\{0\}$
the finite sequence $s= (x,y,x,y, \ldots, x, y )$ having $2n$ alternating terms $x$ and $y$, satisfies $|\ext{s}| = 2n$.
\end{itemize}
\end{example}

\begin{lemma}
\label{zigzag:sequence}
Let $(X,<)$ be a linearly ordered set and let  $X=\{y_n:n\in\N\}$ be the faithful enumeration of $X$ such that $y_m<y_n$
whenever $m,n\in\N$ and $m<n$. Define $x_n=y_{n+1}$ for an even $n\in\N$ and $x_n=y_{n-1}$ for an odd $n\in\N$.
Then $X=\{x_m:m\in\N\}$ is the faithful enumeration of $X$ satisfying the following properties:
\begin{itemize}
\item[(i)]
$|\ext{(x_0,x_1,\ldots,x_m)}|=m+1$
for every $m\in\N$;
\item[(ii)] $\max\{x_0,x_1,\dots,x_{2s+1}\}<x_m$ whenever $m,s\in\N$ and $2s+1<m$.
\end{itemize}
\end{lemma}

\begin{lemma}\label{subsequence:has:less:extrems}
Let $(X,<)$ be a linearly ordered set. If $s\in\Seq(X)$, then $|\ext{s'}|\le|\ext{s}|$ for every subsequence $s'$ of $s$.
\end{lemma}

\begin{proof}
Let $s=(x_0, x_1, ..., x_k)$ for some $k\in\N$, and let $s'$ be a subsequence of $s$. Then there exist $k'\in\N$ and natural numbers $i_0, i_1, \ldots, i_{k'}$ such that
$0\leq i_0< i_1< \ldots  < i_{k'}\leq  k$ and $s'=(x_{i_0}, x_{i_1}, \ldots , x_{i_{k'}})$. For each $p,q\in \N$ with $0\le p<q\le k$ choose integers $j_{\max}(p,q)$
and $j_{\min}(p,q)$ such that
\begin{equation}
\label{inequality:j:max:min}
p\le j_{\max}(p,q)\le q,
\ \
p\le j_{\min}(p,q)\le q,
\end{equation}
\begin{equation}
\label{def:j:max:min}
x_{j_{\max}(p,q)}=\max\{x_{p},x_{p+1},\dots,x_{q-1},x_{q}\}
\ \
\mbox{and}
\ \
x_{j_{\min}(p,q)}=\min\{x_{p},x_{p+1},\dots,x_{q-1},x_{q}\}.
\end{equation}
Define a function $\iota:\ext{s'}\to\N$ by
\begin{equation}
\label{def:alpha}
\iota(n)=
\left\{\begin{array}{ll}
0 & \mbox{if $n=0$}\\
k & \mbox{if $n=k'$}\\
j_{\max}(i_{n-1},i_{n+1})
& \mbox{if $n\in\lma{s'}\setminus\{0,k'\}$}\\
j_{\min}(i_{n-1},i_{n+1})
& \mbox{if $n\in\lmi{s'}\setminus\{0,k'\}$}
\end{array}
\right.
\hbox{\hskip30pt}
\mbox{ for each }
n\in\ext{s'}.
\end{equation}

First of all, let us note that
\begin{equation}
\label{alpha:l:inequality}
i_{n-1}<\iota(n)<i_{n+1}
\mbox{ whenever }
n\in \ext{s'}\setminus\{0,k'\}.
\end{equation}
Indeed, if $n\in \lma{s'}\setminus\{0,k'\}$, then $\iota(n)=j_{\max}(i_{n-1},i_{n+1})$ by \eqref{def:alpha}, so
$i_{n-1}\le \iota(n)\le i_{n+1}$ by
\eqref{inequality:j:max:min}
and
$x_{\iota(n)}=\max\{x_{i_{n-1}},\dots, x_{i_n},\dots,x_{i_{n+1}}\}\ge x_{i_n}$ by \eqref{def:j:max:min}.
Since $x_{i_n}>x_{i_{n-1}}$ and $x_{i_n}>x_{i_{n+1}}$,
we conclude that $\iota(n)\not=i_{n-1}$ and $\iota(n)\not=i_{n+1}$. This proves \eqref{alpha:l:inequality}
when $n\in \lma{s'}\setminus\{0,k'\}$. The proof in the case $n\in \lmi{s'}\setminus\{0,k'\}$ is similar.

Next, let us show that $\iota(\ext{s'})\subseteq \ext{s}$.
Let $n\in \ext{s'}$. If $n=0$ or $n=k'$, then $\iota(n)\in\{0,k\}\subseteq \ext{s}$ by \eqref{def:alpha} and Example
\ref{First:example}(a), as $s\in\Seq(X)$. Assume now that $n\in \ext{s}\setminus\{0,k'\}$.
If $n\in \lma{s'}\setminus\{0,k'\}$, then from
\eqref{def:j:max:min}, \eqref{def:alpha}
and \eqref{alpha:l:inequality} we conclude
that
$x_{\iota(n)-1}\le x_{\iota(n)}$
and
$x_{\iota(n)+1}\le x_{\iota(n)}$.
Since $s\in\Seq(X)$, both inequalities must be strict, and so
$\iota(n)\in \lma{s}\subseteq \ext{s}$. The proof in the case $n\in \lmi{s'}\setminus\{0,k'\}$ is similar.

Finally, to finish the proof of our lemma, it remains only to check that $\iota:\ext{s'}\to\ext{s}$
is an injection. Suppose that $n,m$ are integers such that $0\le n<m\le k'$. We need to show that $\iota(n)\not=\iota(m)$.
If $n=0$, then $\iota(n)=\iota(0)=0\le i_{m-1}<\iota(m)$ by
\eqref{def:alpha} and
\eqref{alpha:l:inequality}.
Similarly, if $m=k'$, then $\iota(n)<i_{n+1}\le k=\iota(k')=\iota(m)$. Suppose now that $0< n<m< k'$.
If $m-n\ge 2$, then $\iota(n)<i_{n+1}\le i_{m-1}<\iota(m)$ by \eqref{alpha:l:inequality}, so we are left only with the case
$m=n+1$. From $n,n+1\in \ext(s')$ we get $x_{i_n}\not=x_{i_{n+1}}$.
Without loss of generality, we may assume that $x_{i_n}<x_{i_{n+1}}$,
so that $n\in\lmi{s'}$ and $n+1\in\lma{s'}$. Suppose that $\iota(n)=\iota(m)$.
Since $\iota(n)=j_{\min}(i_{n-1},i_{n+1})$
and $\iota(m)=\iota(n+1)=j_{\max}(i_{n},i_{n+2})$
by \eqref{def:alpha},
from \eqref{def:j:max:min} we get
$\min\{x_{i_{n-1}},\dots, x_{i_n}, \dots, x_{i_{n+1}}\}
=
\max\{x_{i_{n}},\dots, x_{i_{n+1}}, \dots, x_{i_{n+2}}\}$.
This gives
$x_{i_n}=x_{i_{n+1}}$, in contradiction  with $x_{i_n}<x_{i_{n+1}}$.
Therefore, $\iota(n)\not=\iota(m)$.
\end{proof}

\begin{definition}
\label{eta:definition}
For a linearly ordered set $(X,<)$, we define
the function $\eta:F(X)\to\N\subseteq\R^+$ by
\begin{equation}
\label{def:theta}
\eta(w)=|\ext{\support{w}}|=|\ext{(\letter{w}{1},\letter{w}{2},\dots,\letter{w}{l_w})}|
\ \
 \mbox{ for every } w\in F(X).
\end{equation}
\end{definition}
\begin{lemma}\label{a(g):equation:2}
If $(X,<)$ is a linearly ordered set, then
$\eta$
is a seminorm on $F(X)$.
\end{lemma}
\begin{proof}
We should check that $\eta$ satisfies conditions (S1) and (S2) of Definition \ref{definition:of:a:seminorm}.

(S1)
Fix $a,b\in F(X)$. Let $c=ab$.
Then there exist an initial subword $a'$ of $a$ and a final subword $b'$ of $b$ such that
$c=ab=a'b'$ and the last product is cancellation-free.
In particular,
\begin{equation}
\label{equation:sigma:c}
\support{c}=
(\letter{c}{1},\letter{c}{2},\dots,\letter{c}{l_c})=
(\letter{a'}{1},\letter{a'}{2},\dots,\letter{a'}{l_{a'}},
\letter{b'}{1},\letter{b'}{2},\dots,\letter{b'}{l_{b'}}).
\end{equation}
Since $\support{a'}=(\letter{a'}{1},\letter{a'}{2},\dots,\letter{a'}{l_{a'}})\in\Seq(X)$,
both $\letter{a'}{1}$ and $\letter{a'}{l_{a'}}$ are local extrema of
the sequence $\support{a'}$ by Example \ref{First:example}(a). Similarly, since
$\support{b'}=(\letter{b'}{1},\letter{b'}{2},\dots,\letter{b'}{l_{b'}})\in\Seq(X)$,
both $\letter{b'}{1}$ and $\letter{b'}{l_{b'}}$ are local extrema of
the sequence $\support{b'}$. From this and \eqref{equation:sigma:c} we obtain
\begin{equation}
\label{eq:two:cutted:seq:have:more:ext}
\eta(c)=
|\ext{\support{c}}|\leq|\ext{\support{a'}}|+|\ext{\support{b'}}|.
\end{equation}
Since $\support{a'}$ is a subsequence of $\support{a}$ and
$\support{b'}$ is a subsequence of $\support{b}$, Lemma
\ref{subsequence:has:less:extrems}
implies
$|\ext{\support{a'}}|\le|\ext{\support{a}}|=\eta(a)$
 and
$|\ext{\support{b'}}|\le|\ext{\support{b}}|=\eta(b)$.
Combining this with \eqref{eq:two:cutted:seq:have:more:ext}, we obtain
$\eta(c)\le \eta(a)+\eta(b)$.

(S2) For $a\in F(X)$, the sequence $\support{a^{-1}}=(\letter{a}{l_w},\dots,\letter{a}{2},\letter{a}{1})$
coincides with the sequence $\support{a}=(\letter{a}{1}, \letter{a}{2},\dots,\letter{a}{l_w})$ written in the reverse order,
so $\eta(a^{-1})=|\ext{\support{a^{-1}}}|=|\ext{\support{a}}|=\eta(a)$.
\end{proof}

\begin{definition}
For every $w\in F(X)$ let
\begin{equation}
\label{eq:c(g)}
c(w)=(\cancel{w}{w})^{-1}\cdot w\cdot (\cancel{w}{w})
\end{equation}
denote the conjugate of $w$ by $\cancel{w}{w}$; see Definition \ref{def:of:c} for the notation $\cancel{w}{w}$.
\end{definition}

Lemma \ref{a(g):equation:2} gives an estimate $\eta(w^n)\leq n \eta(w)$ from above for every $w\in F(X)$ and each $n\in\N\setminus\{0\}$.
Our next lemma establishes an estimate from below.

\begin{lemma}\label{cyclic:reduction}
Let $(X,<)$ be a linearly ordered set. If $w\in F(X)$ and $l_{c(w)}>1$, then $\eta(w^n)\ge2n$ for every $n\in\N\setminus\{0\}$.
\end{lemma}

\begin{proof}
It follows from \eqref{eq:c(g)} and Definition \ref{def:of:c} that $w=(\cancel{w}{w})\cdot c(w)\cdot (\cancel{w}{w})^{-1}$, this product is cancellation-free
and the power $ c(w)^n$ is cancellation-free as well. Clearly, $w^n=(\cancel{w}{w})\cdot c(w)^n\cdot(\cancel{w}{w})^{-1}$ and the product on the right
is cancellation-free. Since $l_{c(w)}>1$ and the power $ c(w)^n$ is
cancellation-free, $\support{c(w)^n}$ (and consequently, also $\support{w^n}$) contains a subsequence $s= (x,y,x,y, \ldots, x, y )$
having $2n$ alternating terms $x$ and $y$ for some $x,y\in X$ with $x\not=y$. Since $\support{w^n}\in\Seq(X)$, from \eqref{def:theta},
Lemma \ref{subsequence:has:less:extrems} and Example \ref{First:example}(c), we obtain
$\eta(w^n)=|\ext{\support{w^n}}|\ge |\ext{s}|=2n$.
\end{proof}

Our next example shows that the set $\{\eta(w^n):n\in\N\}$ is unbounded precisely when $l_{c(w)}>1$.

\begin{example}\label{SecondZigZag:example}  Let $(X,<)$ be a linearly ordered set. If $w\in F(X)$ and $l_{c(w)}\leq 1$, then $\eta(w^n)=\eta(w)$ for every
$n\in\N\setminus\{0\}$. To prove this, it suffices to show that
$\support{w}=\support{w^n}$ for every $n\in\N\setminus\{0\}$.
Fix $n\in\N\setminus\{0\}$. If $l_{c(w)}=0$, then $w=e$ and
$\support{w}=\support{w^n}=\support{e}=\emptyset$.
Assume now that $l_{c(w)}=1$. Clearly, $\support{c(w)}=\support{c(w^n)}$.
 From the proof of Lemma \ref{cyclic:reduction} we know that $w=(\cancel{w}{w})\cdot c(w)\cdot(\cancel{w}{w})^{-1}$
and
$w^n=(\cancel{w}{w})\cdot c(w)^n\cdot(\cancel{w}{w})^{-1}$, where
both products on the right side are cancellation-free. Since $\support{c(w)}=\support{c(w^n)}$, this shows that $\support{w}=\support{w^n}$.
\end{example}

\section{A separable metric group answering a question of Dominguez and Tarieladze}
\label{Sec:f-prod:seq:not:f-prod:set}

Let $(\nat{\N},<)$ be a copy of the ordered set $(\N,<)$ of all natural numbers. That is, the map $n\mapsto \nat{n}$ that sends each $n\in\N$ to $\nat{n}\in\nat{\N}$ is a bijection between $\N$ and $\nat{\N}$, and the order relation $<$ on $\nat{\N}$ is defined by declaring $\nat{m}<\nat{n}$ if and only if $m<n$ holds ($n,m\in\N$).
We introduced $\nat{\N}$ in order to make a clear distinction between the usual product $mn$ of natural  numbers $m,n\in\N$ and the product $\nat{m}\cdot\nat{n}$ of their copies taken in the free group $F(\nat{\N})$ with the alphabet $\nat{\N}$.

Let $G=F(\nat{\N})^{\N}$. For $z\in\Z^\N$ define $g_z\in G$ by
\begin{equation}
\label{definition:of:g_z}
g_z(i)=\nat{0}^{z(0)}\nat{1}^{z(1)}\dots\nat{i}^{z(i)}
\ \ \mbox{ for }\ \
i\in\N.
\end{equation}
Let $H$ be the subgroup of $G$ generated by the set $\{g_z:z\in\Z^\N\}$. That is,
\begin{equation}
\label{definition:of:H}
H=\{g_{z_0}^{\varepsilon_0}g_{z_1}^{\varepsilon_1}\dots g_{z_s}^{\varepsilon_s}: s\in\N,
z_0,z_1,\dots,z_s\in\Z^\N,
\varepsilon_0,\varepsilon_1,\dots,\varepsilon_s\in\{-1,1\}\}.
\end{equation}

We consider the discrete topology on
$F(\nat{\N})$,  the Tychonoff product topology on
 $G=F(\nat{\N})^{\N}$ and the subspace topology on $H$ inherited from $G$.

\begin{theorem}
\label{TAP:group:with:ap:sequence}
The topological group $H$ satisfies the following conditions:
\begin{itemize}
\item[(i)] $H$ is a separable metric group;
\item[(ii)] $H$ has a linear topology;
\item[(iii)] $H$ contains an $f_\omega$-productive sequence $B=\{b_n:n\in\N\}$;
\item[(iv)] there exists a bijection $\varphi:\N\to\N$ such that the sequence $\left\{b_{\varphi(n)}:n\in\N\right\}$ is not productive in $H$;
\item[(v)] $H$ does not contain any
unconditionally
$f_\omega$-productive
sequence;
that is, $H$ is TAP.
\end{itemize}
\end{theorem}

Items (i) and (ii) are clear from the definition of $H$. Item (iii) will be proved in
Lemma
\ref{H:is:metric:linear:with:ap:seq}, item (iv) will be proved in
Lemma
\ref{f:omega:productive:seq:but:not:f1:prod:set}, and item (v) will be proved in Lemma \ref{H:has:no:AP:sets}.

For $n\in\N$ define $y_n\in\Z^\N$ by
\begin{equation}
\label{definition:of:y_n}
y_n(m)=
\left\{\begin{array}{ll}
1 & \mbox{if $m=n$}\\
0 & \mbox{if $m\not=n$}
\end{array}
\right.
\hskip30pt
\mbox{ for each }
m\in\N.
\end{equation}

\begin{lemma}
\label{H:is:metric:linear:with:ap:seq}
The sequence $\{g_{y_n}:n\in\N\}$ is an $f_\omega$-productive sequence in $H$.
\end{lemma}
\begin{proof}
 From \eqref{definition:of:g_z} and \eqref{definition:of:y_n} we get
\begin{equation}
\label{definition:of:g_{y_m}}
g_{y_n}(i)=\left\{\begin{array}{ll}
                                       e & \mbox{if $i<n$}\\
                                       \nat{n} & \mbox{if $i\geq n$}
                                    \end{array}
                           \right.
\hskip30pt
\mbox{whenever }
i,n\in\N.
\end{equation}
In particular, the sequence $\{g_{y_n}:n\in\N\}$ is faithfully indexed.
Combining \eqref{definition:of:g_{y_m}} with \eqref{definition:of:g_z}, we conclude that
the sequence $\displaystyle\left\{\prod_{n=0}^k g_{y_n}^{z(n)}:k\in\N\right\}$ converges to $g_z\in H$ for every function $z:\N\to\Z$.
\end{proof}

\begin{remark}
 From the proof of Lemma \ref{H:is:metric:linear:with:ap:seq} and \eqref{definition:of:H}, one can easily deduce that $H$ is the {\em smallest\/}
subgroup of $G$ in which the sequence $\{g_{y_n}:n\in\N\}$ is $f_\omega$-productive.
\end{remark}

By Lemma \ref{a(g):equation:1},
$\mu_{\nat{m}}:F(\nat{\N})\to\N$ is a seminorm on  $F(\nat{\N})$  for every $m\in\N$. Similarly, since $(\nat{\N},<)$ is a linearly ordered set,
Lemma \ref{a(g):equation:2} implies that $\eta:F(\nat{\N})\to\N$ is a seminorm on  $F(\nat{\N})$. Our next lemma shows that these seminorms satisfy the key assumption in Lemmas \ref{technical:lemma} and \ref{impossible:case:for:unbounded:seminorm}, thereby clearing the way for applications of these lemmas later in this section.

\begin{lemma}\label{z(h):is:bounded}
Let $h\in H$. Then:
\begin{itemize}
\item[(i)] the set $\{\mu_{\nat{m}}(h(i)):i\in\N\}$ is bounded for every $m\in\N$;
\item[(ii)] the set $\{\eta(h(i)):i\in\N\}$ is bounded.
\end{itemize}
\end{lemma}

\begin{proof}
Let $\nu$ denote either $\eta$ or $\mu_{\nat{m}}$ for some $m\in\N$. By \eqref{definition:of:H}, there exist $s\in\N$, $z_0,z_1,\dots,z_s\in \Z^\N$ and
$\varepsilon_0,\varepsilon_1,\dots,\varepsilon_s\in\{-1,1\}$ such that $h=g_{z_0}^{\varepsilon_0}g_{z_1}^{\varepsilon_1}\dots g_{z_s}^{\varepsilon_s}$.
Since $\nu$ is a seminorm on $F(\nat{\N})$, for every $i\in \N$ we have
$$
\nu(h(i))=
\nu\left(\prod_{k=0}^s g_{z_k}^{\varepsilon_k}(i)\right)
\le
\sum_{k=0}^s \nu(g_{z_k}^{\varepsilon_k}(i))
=
\sum_{k=0}^s \nu(g_{z_k}(i)^{\varepsilon_k})
=
\sum_{k=0}^s \nu(g_{z_k}(i)).
$$
Therefore, in order to prove that the set $\{\nu(h(i)):i\in\N\}$ is bounded,
it remains only to check that the set $\{\nu(g_z(i)):i\in\N\}$ is bounded for every $z\in \Z^\N$. Fix $z\in \Z^\N$.

(i)
Let $m\in\N$. From \eqref{definition:of:g_z} we conclude that $|\mu_{\nat{m}}(g_z(i))|\le |z(m)|$ for every $i\in\N$.

(ii) Let $i\in\N$. By \eqref{definition:of:g_z}, the sequence $\support{g_z(i)}$ is monotonically increasing,
so $\eta(g_z(i))\le 2$ by Example \ref{First:example}(b) and \eqref{def:theta}.
\end{proof}

In the next section we shall prove at the end of the section that $H$ contains no unconditionally $f_\omega$-productive sequences.
At the present stage we can (easily) see that the $f_\omega$-productive sequence $\{g_{y_{\varphi(n)}}:n\in\N\}$
defined in \eqref{definition:of:g_{y_m}} is not even unconditionally $f_1$-productive in $H$.

\begin{lemma}
\label{f:omega:productive:seq:but:not:f1:prod:set}
There exists a bijection $\varphi:\N\to\N$ such that the sequence $\{g_{y_{\varphi(n)}}:n\in\N\}$ is not productive in $H$.
\end{lemma}
\begin{proof}
Let
$\varphi:\N\to\N$ be a bijection defined by $\varphi(i)=i+1$ for an even $i\in\N$ and $\varphi(i)=i-1$ for an odd $i\in\N$.
Observe that the sequence $\left\{\prod_{i=0}^k g_{y_{\varphi(i)}}:k\in\N\right\}$ converges to the element $a\in G$ such that
$$
a(2l+1)=\nat{1}\cdot\nat{0}\cdot\nat{3}\cdot\nat{2}\cdot\ldots\cdot\nat{2l+1}\cdot\nat{2l}
$$
for every $l\in\N$. Combining this with \eqref{def:theta} and Lemma \ref{zigzag:sequence}(i), we conclude that
$$
\eta(a(2l+1))=|\ext{\support{a(2l+1)}}|=
\ext{(\nat{1},\nat{0},\nat{3},\nat{2},\dots, \nat{2l+1},\nat{2l})}
=
\ext{(1,0,3,2,\dots,2l+1, 2l)}
=
2l+2
$$
for every $l\in\N$. That is, the set $\{\eta(a(i)):i\in\N\}$ is unbounded, and so $a\not\in H$ by Lemma \ref{z(h):is:bounded}(ii).
\end{proof}

\begin{remark}
Dominguez and Tarieladze asked the following question in \cite[Remark 2.5]{DT-private}. Assume that $A=\{a_n:n\in\N\}$ is an 
$f_1^\star$-productive 
sequence in a topological group $H$. Must then the sequence $\{\prod_{i=0}^n a_{\varphi(i)}:n\in\N\}$
converge for every bijection $\varphi:\N\to\N$? As Lemmas \ref{H:is:metric:linear:with:ap:seq}
and \ref{f:omega:productive:seq:but:not:f1:prod:set} show, if one takes $a_n=g_{y_n}$ for every $n\in\N$,
then the resulting sequence provides a 
strong
counter-example to
this question even in a separable metric group $H$ (having a linear topology as well).
\end{remark}

\section{Proof of item (v) of Theorem \ref{TAP:group:with:ap:sequence}}
\label{proof:of:item:v}

For every $j\in\N$, let $\Delta_j:F(\nat{\N})\to\Z$ denote the unique homomorphism extending the function $\delta_j:\nat{\N}\to\Z$ defined by
\begin{equation}
\label{delta:function}
\delta_j(\nat{m})=
\left\{\begin{array}{ll}
1 & \mbox{if $m=j$}\\
0 & \mbox{if $m\not=j$}
\end{array}
\right.
\hskip30pt
\mbox{ for each }
m\in\N.
\end{equation}

\begin{lemma}\label{varphi_n:the:same:from:n}
If $j\in\N$ and $h\in H$, then
\begin{equation}\label{eq:varphi:in:H}
\Delta_j(h(i))=\left\{\begin{array}{ll}
                                       0 & \mbox{if $i<j$}\\
                                       \Delta_j(h(j)) & \mbox{if $i\geq j$}
                                    \end{array}
                           \right.
\hskip30pt
\mbox{ for each }
i\in\N.
                           \end{equation}
\end{lemma}
\begin{proof}
Since $\Delta_j$ is a homomorphism, for every $z\in \Z^\N$, from \eqref{definition:of:g_z} and \eqref{delta:function} we obtain
\begin{equation}
\label{eq:36}
\Delta_j(g_z(i))=\left\{\begin{array}{ll}
                                       0 & \mbox{if $i<j$}\\
                                       z(j) & \mbox{if $i\geq j$}
                                    \end{array}
                           \right.
\hskip30pt
\mbox{ for each }
i\in\N.
\end{equation}
By \eqref{definition:of:H}, there exist $s\in\N$,
$z_0,z_1,\dots,z_s\in \Z^\N$ and $\varepsilon_0,\varepsilon_1,\dots,\varepsilon_s\in\{-1,1\}$ such that
$h=g_{z_0}^{\varepsilon_0}g_{z_1}^{\varepsilon_1}\dots g_{z_s}^{\varepsilon_s}$.
Since $\Delta_j$ is a homomorphism,
from \eqref{eq:36}
we get
$$
\Delta_j(h(i))=
\Delta_j\left(\prod_{n=0}^s g_{z_n}^{\varepsilon_n}(i)\right)
=
\sum_{n=0}^s \varepsilon_n\Delta_j\left(g_{z_n}(i)\right)
=
\left\{\begin{array}{ll}
                                       0 & \mbox{if $i<j$}\\
                                       \sum_{n=0}^s \varepsilon_n z_n(j) & \mbox{if $i\geq j$.}
                                    \end{array}
                           \right.
$$
This yields \eqref{eq:varphi:in:H}.
\end{proof}

Define
\begin{equation}
\label{definition:of:t_a}
t_a=\min\{i\in\N: a(i)\not=e\}
\ \ \mbox{ for each }\ \
a\in H\setminus\{e\}.
\end{equation}
\begin{lemma}\label{lc(b(i))<=1}
Suppose that $z\in\N\setminus\{0\}$, $a\in H\setminus\{e\}$,  $i\in\N$, $i\ge t_a$ and $l_{c(a(k))}\le 1$ for every
$k\in\N$. Then:
\begin{itemize}
\item[(i)] $\mu_{\nat{t_a}}(a^z(i))\ge z$;
\item[(ii)] $\mu_{\nat{j}}(a^z(i))\le \mu_{\nat{j}}(a(i))$ whenever $j\in \N$ and $j\not=t_a$.
\end{itemize}
\end{lemma}

\begin{proof}
Note that
\begin{equation}
\label{eq:38}
a(s)=d_s c(a(s)) d_s^{-1}
\ \ \mbox{ for every } s\in\N,
 \mbox{ where } d_s=\cancel{a(s)}{a(s)}.
\end{equation}
Since each $\Delta_{n}:F(\nat{\N})\to\Z$ is a homomorphism and $\Z$ is abelian, from \eqref{eq:38} we obtain
\begin{equation}\label{chi(a(i))=chi(c(a(i)))}
\Delta_{n}(a(s))=
\Delta_{n}(c(a(s)))
\ \ \mbox{ whenever }\ \
s,n\in\N.
\end{equation}

 From \eqref{definition:of:t_a} and \eqref{eq:38}, it follows that $c(a(t_a))\not=e$. Since $l_{c(a(t_a))}\le 1$,
we conclude that $c(a(t_a))$ is a monom; that is,
$c(a(t_a))=\nat{n}^{p}$ for some $n\in\N$ and $p\in\Z\setminus\{0\}$. Now
\eqref{chi(a(i))=chi(c(a(i)))} yields $\Delta_n(a(t_a))=\Delta_{n}(c(a(t_a)))=\Delta_{n}(\nat{n}^{p})=p\neq
0$. From this and Lemma \ref{varphi_n:the:same:from:n}, we conclude that $n\le t_a$.
If $n<t_a$, then $a(n)=e$ by \eqref{definition:of:t_a}, and applying Lemma \ref{varphi_n:the:same:from:n} once again, we get
$\Delta_{n}(a(t_a))=\Delta_{n}(a(n))=\Delta_{n}(e)=0$, a contradiction. This shows that $n=t_a$. In particular,
$\Delta_{t_a}(a(t_a))=p$. From $i\ge t_a$,
\eqref{chi(a(i))=chi(c(a(i)))} and Lemma \ref{varphi_n:the:same:from:n},
we get $\Delta_{t_a}(c(a(i)))=\Delta_{t_a}(a(i))=\Delta_{t_a}(a(t_a))=p$.
Since $l_{c(a(i))}\le 1$, this yields $c(a(i))=\nat{t_a}^{p}$, and so
\begin{equation}
\label{eq:a(i)}
a(i)=d_i c(a(i)) d_i^{-1}=d_i\nat{t_a}^{p} d_i^{-1}
\end{equation}
by \eqref{eq:38}. Since $d_i=\cancel{a(i)}{a(i)}$ is an initial
subword of $a(i)$, this implies that the last product in
\eqref{eq:a(i)} is cancellation-free. Furthermore,
$$
a^z(i)=a(i)^z=(d_i \nat{t_a}^{p} d_i^{-1})^z=d_i \nat{t_a}^{pz} d_i^{-1},
$$
and the product on the right is cancellation-free as well. In
particular,
\begin{equation}
\label{mu:eq:power:q}
\mu_{\nat{j}}(a^z(i))=\max\left\{\mu_{\nat{j}}\left(d_i\right),
\mu_{\nat{j}}\left(\nat{t_a}^{pz}\right),\mu_{\nat{j}}\left(d_i^{-1}\right)\right\}
=
\max\left\{    \mu_{\nat{j}}  \left(d_i \right), \mu_{\nat{j}} \left( \nat{t_a}^{pz} \right) \right   \}
\mbox{ for every }
j\in\N.
\end{equation}
Since $|p|\ge 1$, this implies
$\mu_{\nat{t_a}}(a^z(i))\ge \mu_{\nat{t_a}}\left(\nat{t_a}^{pz}\right)=|pz|\ge |z|=z$. This finishes the proof of (i).

Assume now that $j\in\N$ and $j\not=t_a$. Since $d_i$ is an initial subword of $a(i)$, we have
$\mu_{\nat{j}}(d_i)\le \mu_{\nat{j}}(a(i))$. Since $\mu_{\nat{j}}\left(\nat{t_a}^{pz}\right)=0$, from
\eqref{mu:eq:power:q}
we obtain $\mu_{\nat{j}}(a^z(i))= \mu_{\nat{j}}(d_i)\le \mu_{\nat{j}}(a(i))$. This proves (ii).
\end{proof}

\begin{lemma}
\label{thin:sets}
Let $\{h_n:n\in\N\}$ be a faithfully indexed sequence of elements of $H\setminus\{e\}$ such that $l_{c(h_n(i))}\le 1$ whenever
$i,n\in\N$. Then $\{h_n:n\in\N\}$ is not
unconditionally $f_\omega$-productive
in $H$.
\end{lemma}

\begin{proof}
We consider two cases, with the easiest one coming first.

\medskip
{\em Case 1\/}. {\sl The set $\{\mu_{\nat{j}}(h_n(i)):i,n\in \N\}$ is unbounded for some $j\in\N$.\/}
It follows from Lemma \ref{z(h):is:bounded}(i) that  the set $\{n\in\N: \exists\ i\in\N\ \mu_{\nat{j}}(h_n(i))\ge r\}$ must be infinite for each $r\in \R$.
This means that the condition (ii) of Lemma \ref{impossible:case:for:unbounded:seminorm} holds, if one takes
$I=\N$, $f=f_\omega$, $D=F(\nat{\N})$ and $\nu=\mu_{\nat{j}}$. Lemma \ref{z(h):is:bounded}(i) implies that the condition (i) of Lemma
\ref{impossible:case:for:unbounded:seminorm} holds as well. Finally, $\mu_{\nat{j}}$ is a seminorm on $F(\nat{\N})$ by Lemma \ref{a(g):equation:1}.
Applying Lemma \ref{impossible:case:for:unbounded:seminorm}, we conclude that
$\{h_n:n\in\N\}$ is not an $f_\omega$-productive sequence in $H$.

\medskip
{\em Case 2\/}. {\sl The set $\{\mu_{\nat{j}}(h_n(i)):i,n\in \N\}$ is  bounded
for each $j\in\N$.\/} This allows us to fix a function $\psi:\N\to\N$ such that
\begin{equation}
\label{eq:uniform:bound}
\mu_{\nat{j}}(h_n(i))\le \psi(j)
\mbox{ whenever }
i,j,n\in \N.
\end{equation}

\medskip
{\em Subcase 2a\/}. {\sl  The set $E_j=\{n\in\N: t_{h_n}=j\}$ is infinite for some
$j\in\N$\/}.
From
Lemma \ref{lc(b(i))<=1}(i) it follows that $\mu_j(h_n(j))\ge 1$ for each $n\in E_j$. In particular, $h_n(j)\not=e$ for $n\in E_j$.
Since $F(\nat{\N})$ has the discrete topology and $E_j$ is infinite, we conclude that the sequence $\{h_n(j):n\in\N\}$ does
not converge to $e$ in $F(\nat{\N})$. It follows that the sequence $\{h_n:n\in\N\}$ does not converge to $e$ in $H$.
Therefore, this sequence is not $f_\omega$-productive in $H$ by
Lemma \ref{AP:is:null:sequence}.

\medskip
{\em Subcase 2b\/}. {\sl The set $\{n\in\N: t_{h_n}=j\}$ is finite for every $j\in\N$\/}. Then the set $T=\{t_{h_n}:n\in\N\}$ must be infinite, so
we can use Lemma \ref{zigzag:sequence} to fix a faithful enumeration $T=\{p_m:m\in\N\}$ of $T$ such that
\begin{equation}
\label{eq:number:of:extremums}
|\ext{(p_0,p_1,\ldots,p_m)}|=m+1
\ \ \mbox{ for every }\ \
m\in\N,
\end{equation}
and
\begin{equation}
\label{eq:max:p_s}
\max\{p_0,p_1,\dots,p_{2s+1}\}<p_m
\mbox{ whenever }
m,s\in\N
\mbox{ and }
2s+1<m.
\end{equation}
Let $\varphi:\N\to\N$ be the (unique) injection such that
\begin{equation}
\label{def:p_m}
p_m=t_{h_{\varphi(m)}}
\ \ \mbox{ for every }\ \
m\in\N.
\end{equation}
According to Lemma \ref{AP:set:has:AP:subsets}, in order to show that $\{h_n:n\in\N\}$ is not unconditionally $f_\omega$-productive
in $H$, it suffices to prove that the sequence $\{h_{\varphi(m)}:m\in\N\}$ is not $f_\omega$-productive in $H$. To achieve this,
we are going to make use of Lemma
\ref{technical:lemma}. For every $m\in\N$, let
\begin{equation}
\label{eq:def:a_m}
a_m=h_{\varphi(m)},
\end{equation}
\begin{equation}
\label{eq:def:z_m}
z_m=2\psi(p_m)+1
\end{equation}
and
\begin{equation}
\label{eq:i_m}
i_m=\max\{p_0,p_1,\ldots,p_m\}.
\end{equation}
 From \eqref{def:p_m}
and
\eqref{eq:def:a_m} it follows that
\begin{equation}
\label{eq:p_m:and:t_a_m}
p_m=t_{a_m}
\ \ \mbox{ for every }\ \
m\in\N.
\end{equation}
\begin{claim}
\label{we:can:apply:technical:lemma}
$I=\N$, $g=f_\omega$, $D=F(\nat{\N})$, $\nu=\eta$,
$A=\{a_m:m\in\N\}$, $M=\{2s+1:s\in\N\}$,
$\{z_m:m\in\N\}\subseteq\Z$
and $\{i_m:m\in\N\}\subseteq \N$ satisfy the assumptions of Lemma \ref{technical:lemma}.
\end{claim}

\begin{proof}
Note that $\nu=\eta$ is a seminorm on $D$ by Lemma
\ref{a(g):equation:2}. Lemma \ref{z(h):is:bounded}(ii) guarantees the validity of another assumption of Lemma \ref{technical:lemma}.
It remains only to check that conditions (i$_m$), (ii$_m$) and (iii$_m$)
of Lemma \ref{technical:lemma} hold for every $m\in\N$.

(i$_m$) $|z_m|< \omega=f_\omega(m)=g(m)$.

(ii$_m$) Assume that $k\in M$ and $0\le k<m$. Thus $k=2s+1$ for some $s\in \N$. From
\eqref{eq:max:p_s}, \eqref{eq:i_m} and \eqref{eq:p_m:and:t_a_m}, we get $i_k=i_{2s+1}<p_m=t_{a_m}$.
Thus, $a_m(i_k)=e$ by \eqref{definition:of:t_a}.

(iii$_m$)
 Assume that $m\in M$; that is, $m=2s+1$ for some $s\in\N$. In particular, $m\ge 1$.
For every $j\in\N$ with $j\le m$ define $w_j=a_j^{z_j}(i_m)\in F(\nat{\N})$
and $x_j=\nat{p_j}\in\nat{\N}$. We must prove that $\eta\left(\prod_{j=0}^m w_j\right)\ge m$.

First, let us show that the set $X=\nat{\N}$, the integer $m$, elements $w_0,w_1,\dots,w_m\in F(\nat{\N})$
and letters $x_0,x_1,\dots,x_m\in\nat{\N}$ satisfy the assumption of
Lemma
\ref{real:final:step}.
Indeed, let $j,k\in\N$, $j\le m$, $k\le m$ and $j\not = k$. From $j\le m$,
\eqref{eq:p_m:and:t_a_m} and \eqref{eq:i_m}, we get $t_{a_j}=p_j\le i_m$.
Therefore,
\begin{equation}
\label{eq:40}
\mu_{x_j}(w_j)=\mu_{\nat{p_j}}(w_j)=
\mu_{\nat{p_j}}\left(a_j^{z_j}(i_m)\right)
\ge z_j
=
 2\psi(p_j)+1> 2\psi(p_j)
\end{equation}
by
Lemma \ref{lc(b(i))<=1}(i) and \eqref{eq:def:z_m}. Note that $j\not=k$ and \eqref{eq:p_m:and:t_a_m} imply
$p_j\not=p_k=t_{a_k}$, and so
\begin{equation}
\label{eq:41}
\mu_{x_j}(w_k)=\mu_{\nat{p_j}}(w_k)=
\mu_{\nat{p_j}}(a_k^{z_k}(i_m))\le
\mu_{\nat{p_j}}(a_k(i_m))
\end{equation}
by Lemma \ref{lc(b(i))<=1}(ii). Since $a_k=h_{\varphi(k)}$ by \eqref{eq:def:a_m}, from \eqref{eq:uniform:bound} we obtain
$$
\mu_{\nat{p_j}}(a_k(i_m))=\mu_{\nat{p_j}}(h_{\varphi(k)}(i_m))\le \psi(p_j).
$$
Combining this with \eqref{eq:40} and \eqref{eq:41}, we conclude that $\mu_{x_j}(w_j)>2\psi(p_j)\ge 2 \mu_{x_j}(w_k)$.

Second, from Lemma \ref{real:final:step} it follows that the sequence $(x_0,x_1,\dots,x_m)$ is a subsequence
of the sequence $\support{v}$, where $v=\prod_{j=0}^m w_j$.
Since $\support{v}\in\Seq(\nat{\N})$,
Lemma \ref{subsequence:has:less:extrems} and \eqref{def:theta} yield
$$
\eta(v)=|\ext{\support{v}}| \ge |\ext{(x_0,x_1,\dots,x_m)}|.
$$
Since
$$
|\ext{(x_0,x_1,\dots,x_m)}|
=
|\ext{(\nat{p_0},\nat{p_1},\dots,\nat{p_m})}|
=
|\ext{(p_0,p_1,\dots,p_m)}|
=m+1
$$
by
\eqref{eq:number:of:extremums}, this shows that $\eta\left(\prod_{j=0}^m w_j\right)=\eta(v)\ge m$.
\end{proof}

 From Claim \ref{we:can:apply:technical:lemma} and Lemma \ref{technical:lemma}, it now follows that $A$ is not an $f_\omega$-productive sequence in $H$.
\end{proof}

\begin{lemma}
\label{H:has:no:AP:sets} $H$ does not contain an
unconditionally
$f_\omega$-productive
sequence.
\end{lemma}
\begin{proof}
Let $\{h_n:n\in\N\}$ be an arbitrary faithfully indexed sequence of elements if $H$.
We need to prove that $\{h_n:n\in\N\}$ is not
unconditionally
$f_\omega$-productive
in $H$. Without
loss of generality, we may assume that $h_n\neq e$ for every $n\in\N$.

\medskip
{\em Case 1\/}. {\sl The set $K=\{n\in \N: \exists\ i\in\N\ l_{c(h_n(i))}\ge 2\}$ is finite\/}. Let $k=\max K$. By Lemma \ref{thin:sets},
the sequence
$\{h_n:n\in\N, n>k\}$ is not
unconditionally
$f_\omega$-productive
in $H$.
Clearly, $\{h_n:n\in\N\}$ cannot be
unconditionally
$f_\omega$-productive
in $H$
either.

\medskip
{\em Case 2\/}. {\sl The set $K$ is infinite\/}.
It suffices to check that $I=\N$, $f=f_\omega$, $D=F(\nat{\N})$, $\nu=\eta$ and $\{h_n:n\in\N\}$ satisfy the assumptions of Lemma
\ref{impossible:case:for:unbounded:seminorm}. Recall that  $\eta$ is a seminorm on $D$ by Lemma
\ref{a(g):equation:2}. Condition (i) of Lemma \ref{impossible:case:for:unbounded:seminorm} is
satisfied by Lemma \ref{z(h):is:bounded}(ii). Lemma \ref{cyclic:reduction} implies that $K\subseteq N_r$ for every $r\in\R$.
Therefore, condition (ii) of Lemma \ref{impossible:case:for:unbounded:seminorm} is satisfied as well.
\end{proof}

\bigskip

We finish this paper with two open questions which are motivated by our Theorem \ref{TAP:group:with:ap:sequence}.

\begin{question}
Does there exist a (separable metric) group which contains an $f_\omega$-productive sequence
but does not contain any unconditionally $f_\omega$-Cauchy productive sequence?
\end{question}
\begin{question}
Does there exist a (separable metric) group $G$ such that:
\begin{itemize}
\item[(i)] $G$ contains an $f_\omega$-productive sequence;
\item[(ii)] for every faithfully indexed sequence $\{b_n:n\in\N\}$ of elements  of $G$ there
exists a bijection $\varphi:\N\to\N$ such that the sequence $\left\{b_{\varphi(n)}:n\in\N\right\}$ is not
(Cauchy)
productive?
\end{itemize}
\end{question}

\bigskip
\noindent
{\bf Historical comment:}
Sections 1--10 are the result of re-arrangement and substantial
improvement of (some of) the
material from an earlier preprint
\cite{DSS}. Sections 11--15 are completely new.

\bigskip
\noindent
{\bf Acknowledgement:} We would like to thank cordially Vaja Tarieladze for fruitful exchange of ideas that inspired
us to prove Theorem \ref{NSS:is:NACP} and Corollary \ref{weil:NSS:iff:TAP}, as well as for providing us with the reference \cite{Drewnowski}.

\end{document}